\documentclass[12pt]{article}
\pdfoutput=1

\usepackage{amsmath}
\usepackage{mathtools}
\usepackage{amsthm}		%
\usepackage{amssymb}
\usepackage{enumerate}
\usepackage{cases}
\usepackage{multirow}
\usepackage{subfigure}
\usepackage{graphicx}
\usepackage{appendix}
\usepackage{pgf,tikz}
\usetikzlibrary{arrows}
\usepackage{afterpage}
\graphicspath{{./}}

\DeclareMathOperator{\diam}{diam}
\DeclareMathOperator{\dist}{dist}

\DeclareMathOperator{\supp}{supp}

\begin{document}
\newcommand{\done}[2]{\dfrac{d {#1}}{d {#2}}}
\newcommand{\donet}[2]{\frac{d {#1}}{d {#2}}}
\newcommand{\pdone}[2]{\dfrac{\partial {#1}}{\partial {#2}}}
\newcommand{\pdonet}[2]{\frac{\partial {#1}}{\partial {#2}}}
\newcommand{\pdonetext}[2]{\partial {#1}/\partial {#2}}
\newcommand{\pdtwo}[2]{\dfrac{\partial^2 {#1}}{\partial {#2}^2}}
\newcommand{\pdtwot}[2]{\frac{\partial^2 {#1}}{\partial {#2}^2}}
\newcommand{\pdtwomix}[3]{\dfrac{\partial^2 {#1}}{\partial {#2}\partial {#3}}}
\newcommand{\pdtwomixt}[3]{\frac{\partial^2 {#1}}{\partial {#2}\partial {#3}}}
\newcommand{\bs}[1]{\mathbf{#1}}
\newcommand{\bx}{\mathbf{x}}
\newcommand{\by}{\mathbf{y}}
\newcommand{\bd}{\mathbf{d}} 
\newcommand{\bn}{\mathbf{n}} 
\newcommand{\bP}{\mathbf{P}} 
\newcommand{\bp}{\mathbf{p}} 
\newcommand{\ol}[1]{\overline{#1}}
\newcommand{\rf}[1]{(\ref{#1})}
\newcommand{\xt}{\mathbf{x},t}
\newcommand{\hs}[1]{\hspace{#1mm}}
\newcommand{\vs}[1]{\vspace{#1mm}}
\newcommand{\eps}{\varepsilon}
\newcommand{\ord}[1]{\mathcal{O}\left(#1\right)} 
\newcommand{\oord}[1]{o\left(#1\right)}
\newcommand{\Ord}[1]{\Theta\left(#1\right)}
\newcommand{\PhiF}{\Phi_{\rm freq}}
\newcommand{\real}[1]{{\rm Re}\left[#1\right]} 
\newcommand{\im}[1]{{\rm Im}\left[#1\right]}
\newcommand{\hsnorm}[1]{||#1||_{H^{s}(\bs{R})}}
\newcommand{\hnorm}[1]{||#1||_{\tilde{H}^{-1/2}((0,1))}}
\newcommand{\norm}[2]{\left\|#1\right\|_{#2}}
\newcommand{\normt}[2]{\|#1\|_{#2}}
\newcommand{\on}[1]{\Vert{#1} \Vert_{1}}
\newcommand{\tn}[1]{\Vert{#1} \Vert_{2}}
\newcommand{\ts}{\tilde{s}}
\newcommand{\tGamma}{{\tilde{\Gamma}}}
\newcommand{\darg}[1]{\left|{\rm arg}\left[ #1 \right]\right|}
\newcommand{\bnabla}{\boldsymbol{\nabla}}
\newcommand{\dive}{\boldsymbol{\nabla}\cdot}
\newcommand{\curl}{\boldsymbol{\nabla}\times}
\newcommand{\Phixy}{\Phi(\bx,\by)}
\newcommand{\PhiOxy}{\Phi_0(\bx,\by)}
\newcommand{\dxPhixy}{\pdone{\Phi}{n(\bx)}(\bx,\by)}
\newcommand{\dyPhixy}{\pdone{\Phi}{n(\by)}(\bx,\by)}
\newcommand{\dxPhiOxy}{\pdone{\Phi_0}{n(\bx)}(\bx,\by)}
\newcommand{\dyPhiOxy}{\pdone{\Phi_0}{n(\by)}(\bx,\by)}

\newcommand{\rd}{\mathrm{d}}
\newcommand{\R}{\mathbb{R}}
\newcommand{\N}{\mathbb{N}}
\newcommand{\Z}{\mathbb{Z}}
\newcommand{\C}{\mathbb{C}}
\newcommand{\K}{{\mathbb{K}}}
\newcommand{\ri}{{\mathrm{i}}}
\newcommand{\re}{{\mathrm{e}}} 

\newcommand{\cA}{\mathcal{A}}
\newcommand{\cC}{\mathcal{C}}
\newcommand{\cS}{\mathcal{S}}
\newcommand{\cD}{\mathcal{D}}
\newcommand{\cone}{{c_{j}^\pm}}
\newcommand{\ctwo}{{c_{2,j}^\pm}}
\newcommand{\cthree}{{c_{3,j}^\pm}}

\newtheorem{thm}{Theorem}[section]
\newtheorem{lem}[thm]{Lemma}
\newtheorem{defn}[thm]{Definition}
\newtheorem{prop}[thm]{Proposition}
\newtheorem{cor}[thm]{Corollary}
\newtheorem{rem}[thm]{Remark}
\newtheorem{conj}[thm]{Conjecture}
\newtheorem{ass}[thm]{Assumption}
\newcommand{\tk}{\tilde{k}}
\newcommand{\tR}{\tilde{R}}
\newcommand{\bxi}{\boldsymbol{\xi}}
\newcommand{\be}{\mathbf{e}}
\newcommand{\sD}{\mathsf{D}}
\newcommand{\sN}{\mathsf{N}}
\newcommand{\sE}{\mathsf{E}}
\newcommand{\sS}{\mathsf{S}}
\newcommand{\sT}{\mathsf{T}}
\newcommand{\sH}{\mathsf{H}}
\newcommand{\sI}{\mathsf{I}}
\newcommand{\sP}{\mathsf{P}}
\newcommand{\cT}{\mathcal{T}}%
\newcommand{\dudnjump}{\left[ \pdone{u}{\bn}\right]}
\newcommand{\dudnjumptext}{[ \pdonetext{u}{\bn}]}
\newcommand{\sumpm}[1]{\{\!\{#1\}\!\}}
\newcommand{\uM}{M(u)}
\newcommand{\uMp}{M(u')}
\newcommand{\citep}{\cite}
\newcommand{\citet}{\cite}
\title{A frequency-independent boundary element method for scattering by two-dimensional screens and apertures}

\author{%
{\sc
D.\ P.\ Hewett\thanks{Corresponding author. Email: hewett@maths.ox.ac.uk.  Current address: Mathematical Institute, University of Oxford, UK},
S.\ Langdon\thanks{Email: s.langdon@reading.ac.uk} 
and
S.\ N.\ Chandler-Wilde\thanks{Email: s.n.chandler-wilde@reading.ac.uk}} \\[2pt]
Department of Mathematics and Statistics, University of Reading, UK.
}

\maketitle

\begin{abstract}
{
We propose and analyse a hybrid numerical-asymptotic $hp$ boundary element method for time-harmonic scattering of an incident plane wave by an arbitrary collinear array of sound-soft two-dimensional screens.  Our method uses an approximation space enriched with oscillatory basis functions, chosen to capture the high frequency asymptotics of the solution. 
We provide a rigorous frequency-explicit error analysis which proves that the method converges exponentially as the number of degrees of freedom $N$ increases, and that to achieve any desired accuracy it is sufficient to increase $N$ in proportion to the square of the logarithm of the frequency as the frequency increases (standard boundary element methods require $N$ to increase at least linearly with frequency to retain accuracy). Our numerical results suggest that fixed accuracy can in fact be achieved at arbitrarily high frequencies with a frequency-independent computational cost, when the oscillatory integrals required for implementation are computed using Filon quadrature. We also show how our method can be applied to the complementary ``breakwater'' problem of propagation through an aperture in an infinite sound-hard screen.
}
\\[5mm] \textbf{Keywords}:
{high frequency scattering; hybrid numerical-asymptotic boundary element method; diffraction; screen; strip; aperture; breakwater.}
\end{abstract}
\section{Introduction}
\label{Intro}
The problem of time-harmonic scalar wave scattering of an incident plane wave by a two-dimensional (2D) sound-soft screen, and the related problem of scattering by an aperture in an infinite sound-hard screen, are amongst the most widely studied scattering problems.  
They are the simplest canonical problems that exhibit multiple diffraction, yet have applications in acoustics (see, e.g., \cite{To:89,LiWo:05}), electromagnetics (see, e.g., \cite{DavisChew08,VoLy:11}) and water waves (the ``breakwater'' problem, see e.g.\ \cite{BiPoSt:00}, \citet[chapter~4.7]{LiMc:01}).  
In this paper, we propose a numerical method (supported by a complete analysis) that we believe to be the first method of any kind (numerical or analytical) for this problem that is provably effective at all frequencies. Precisely, we prove that increasing the number of degrees of freedom in proportion to the square of the logarithm of the frequency is sufficient to maintain any desired accuracy as the frequency increases. Moreover, our numerical experiments suggest that, in practice, with a fixed number of degrees of freedom the accuracy stays fixed or even improves as frequency increases.

We consider the 2D problem of scattering of a time-harmonic incident plane wave
\[ %
  u^i(\bx) :=\re^{\ri  k \bx\cdot \bd}, \qquad \bx=(x_1,x_2)\in\mathbb{R}^2, %
\] %
where $k>0$ is the wavenumber (proportional to frequency) and $\bd=(d_1,d_2)\in\mathbb{R}^2$ is a unit direction vector, by a sound-soft screen $\Gamma$, a bounded and relatively open non-empty subset of $\Gamma_\infty:=\{ \bx\in\R^2:x_2=0\}$. 
Here the propagation domain is the set $D:=\R^2\setminus \overline\Gamma$, where $\overline \Gamma$ denotes the closure of $\Gamma$. 
We also consider the complementary problem of scattering due to an aperture $\Gamma$ in a sound-hard screen occupying $\Gamma_\infty\setminus\overline{\Gamma}$. In this case the propagation domain is $D':=(\R^2\setminus \Gamma_\infty)\cup\Gamma$.  %
In both cases we assume 
that $\Gamma$ is a union of a finite number of disjoint open intervals, i.e.
\begin{align}
\label{eqn:GammaDef}
\Gamma=\{ (x_1,0)\in\R^2:x_1\in\tGamma\},
\qquad
\tGamma = \bigcup_{j=1}^{n_i} (s_{2j-1},s_{2j}),
\end{align}
where $n_i\geq 1$ is the number of intervals making up $\Gamma$, and \mbox{$0=s_1<s_2<\ldots<s_{2n_i}=L:=\diam{\Gamma}$}. 
(In the case $n_i=1$, $L$ simply represents the length of the screen.) For each $j=1,\ldots,n_i$ we set $\Gamma_j:=(s_{2j-1},s_{2j})\times\{0\}\subset\R^2$ and $L_j:=s_{2j}-s_{2j-1}$.

Our analysis is in the context of Sobolev spaces, the notation and basic definitions for which are set out in~\S\ref{sec:sobolev}. In what follows, let $U^+$ and $U^-$ denote respectively the upper and lower half-planes, i.e., $U^+:= \{\bx\in\R^2:x_2>0\}$ and $U^- := \R^2\setminus \overline{U^+}$, and let $\gamma^\pm$ and $\partial_\bn^\pm$ denote respectively the Dirichlet and Neumann traces from $U^\pm$ onto $\Gamma_\infty$ (defined precisely in~\S\ref{sec:sobolev}).

For the screen scattering problem, the boundary value problem (BVP) to be solved is
\begin{defn}[Problem $\sP$]
Find $u\in C^2\left(D\right)\cap  W^1_{\mathrm{loc}}(D)$ such that
\begin{eqnarray}
 & \Delta u+k^2u  =  0, \quad \mbox{in }D, & \label{eqn:HE1} \\
 & u  = 0, \quad \mbox{on }\Gamma,&\label{eqn:bc1}
\end{eqnarray}
and the scattered field $u^s:=u-u^i$ satisfies the Sommerfeld radiation condition (see, e.g., \citet[(2.9)]{ChGrLaSp:11}). %
By \rf{eqn:bc1} we mean, precisely, that $\gamma^\pm (\chi u)\vert_\Gamma=0$, for every $\chi\in C_0^\infty(\R^2)$ (where for integer $n\geq1$, $C^\infty_{0}(\R^n)$ is the set of those $u \in C^\infty(\R^n)$ that are compactly supported).%
\end{defn}

In Figure~\ref{fig:screen_domain} we plot the total field $u$ for Problem~$\sP$ for a particular scattering configuration (with $\Gamma$ defined precisely in~\S\ref{sec:num}), for two different values of $k$.
\begin{figure}[p]
\begin{center}
\subfigure{\includegraphics[width=11.2cm]{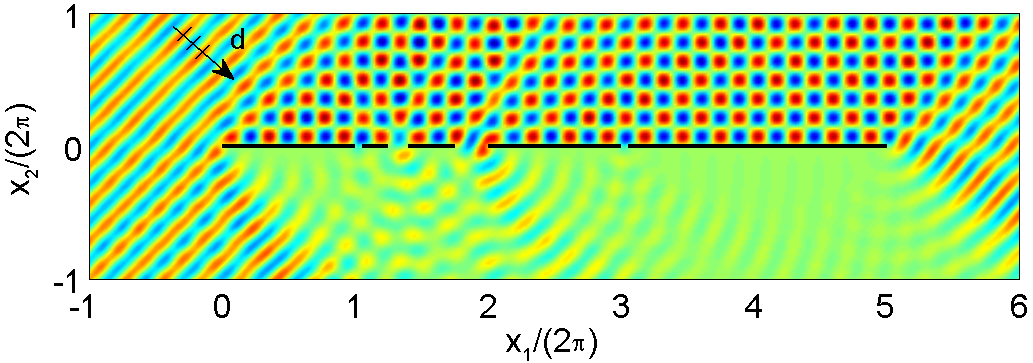}}
\vs{1}
\subfigure{\includegraphics[width=11.2cm]{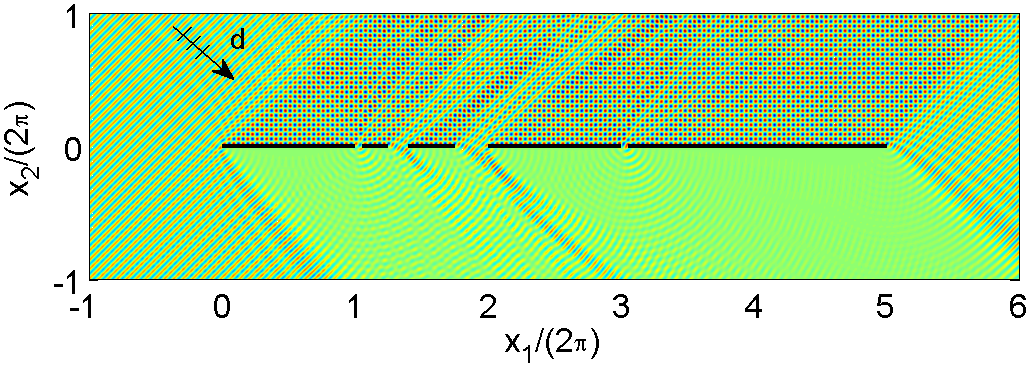}}
\end{center}
\vs{-4}
\caption{Total field $u$, solving Problem~$\sP$, for $\bd=(1/\sqrt{2}, -1/\sqrt{2})$ with $k=5$ (upper) and $k=20$ (lower).}
\label{fig:screen_domain}
\end{figure}

For the aperture scattering problem, the BVP to be solved is (for definiteness we assume in this case that $d_2<0$, so the incident wave arrives from the region $x_2>0$, i.e.\ from above the screen):
\begin{defn}[Problem $\sP'$]
Find $u'\in C^2\left(D'\right)\cap  W^1_{\mathrm{loc}}(D')$ such that
\begin{eqnarray}
 & \Delta u'+k^2u'  =  0, \quad \mbox{in }D', & \label{eqn:HE2} \\
 & \pdonetext{u'}{\bn}  = 0, \quad \mbox{on }\Gamma_\infty\setminus\overline\Gamma,&\label{eqn:bc2}
\end{eqnarray}
and %
\begin{align}
\label{eqn:ud}
u^d(\bx) :=
\begin{cases}
u'(\bx)-(u^i(\bx)+u^r(\bx)), & \bx\in U^+,\\
u'(\bx), & \bx\in U^-,
\end{cases}
\end{align}
satisfies the Sommerfeld radiation condition, where
\[ %
  u^r(\bx) :=\re^{\ri  k \bx\cdot \bd'}, \qquad \bx\in\mathbb{R}^2, %
\qquad \mbox{ with }\bd':=(d_1,-d_2).
\] %
By \rf{eqn:bc2} we mean, precisely, that $\partial_\bn^\pm (\chi u')\vert_{\Gamma_\infty\setminus\overline\Gamma}=0$, for every $\chi\in C_0^\infty(\R^2)$. 
\end{defn}

The solutions to Problems $\sP$ and $\sP'$ are very closely related: as will be made explicit in Theorem~\ref{DirExUn}, once the solution to one problem is known, the solution to the other follows immediately (this is a manifestation of Babinet's principle). 
In Figure~\ref{fig:aperture_domain} we plot the total field $u'$ for the aperture Problem $\sP'$ for the same $\Gamma$, $\bd$, and $k$ as in Figure~\ref{fig:screen_domain}. 
We remark that the field $u^d$ in Problem $\sP'$ can be thought of as a `diffracted' field, being the result of subtracting from the total field $u'$ the incident and reflected plane waves $u^i$ and $u^r$ in the region $x_2>0$. 
\begin{figure}[p]
\begin{center}
\subfigure{\includegraphics[width=11.2cm]{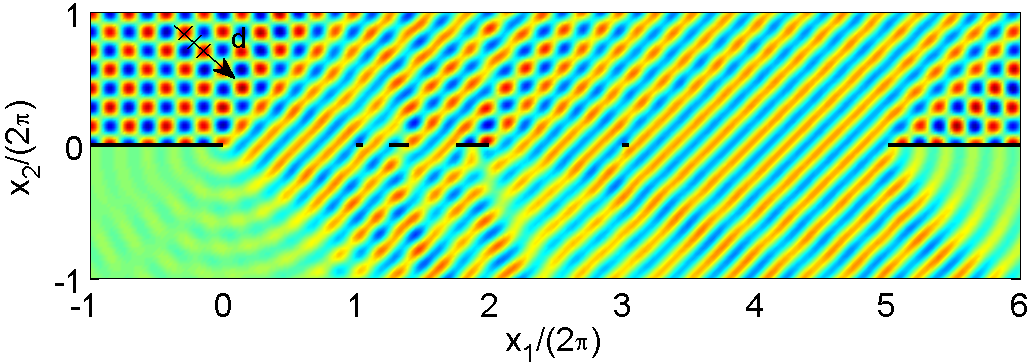}}
\vs{1}
\subfigure{\includegraphics[width=11.2cm]{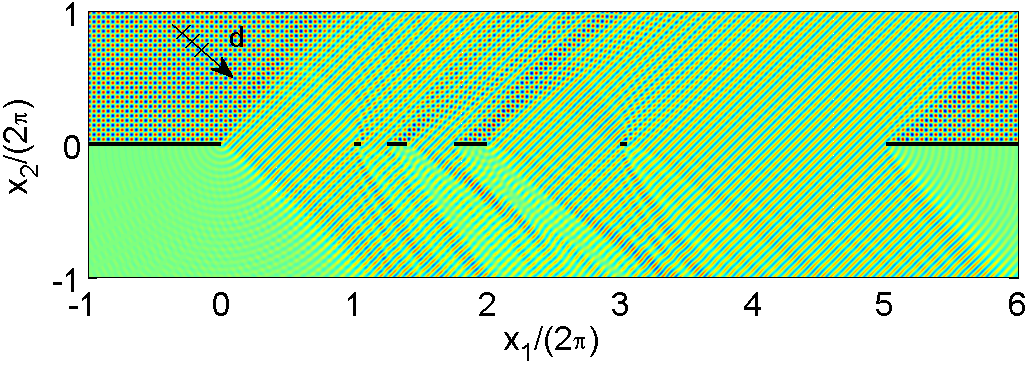}}
\end{center}
\vs{-4}
\caption{Total field $u'$, solving Problem~$\sP'$, for $\bd=(1/\sqrt{2}, -1/\sqrt{2})$ with $k=5$ (upper) and $k=20$ (lower).}
\label{fig:aperture_domain}
\end{figure}

Our approach to solving Problems $\sP$ and $\sP'$ is to reformulate the BVPs as a boundary integral equation (BIE) on $\Gamma$ (see~\S\ref{sec:prob}), which we then solve numerically by a hybrid numerical-asymptotic (HNA) Galerkin boundary element method (BEM). %
The key idea of the HNA approach is to use knowledge of the high frequency asymptotic behaviour of the solution on $\Gamma$ to incorporate appropriate oscillations into the approximation space in such a way that only non-oscillatory functions need to be approximated numerically, so as to achieve a good approximation for a relatively small number of degrees of freedom.  This approach has been successfully applied to a range of scattering problems, e.g.\ scattering by smooth convex 2D obstacles \citep{DGS07}, convex polygons \citep{Convex,HeLaMe:11} and non-convex polygons \citep{NonConvex}, and was the subject of the recent survey paper, \citet{ChGrLaSp:11}. 
We believe that the current paper represents the first application of the HNA methodology, supported by a full numerical analysis, to problems of scattering by screens.

While the numerical method we propose closely resembles that proposed in \citet{HeLaMe:11} for scattering by convex polygons, the numerical analysis for the screen problem is significantly more challenging than that in \citet{HeLaMe:11} and other previous work. The key difference here is that, due to the strong singularity induced by the edge of the screen, the solution to our BIE does not lie in $L^2(\Gamma)$ (as is the case for all previous numerical analyses of HNA methods), and thus we must derive regularity estimates, best approximation estimates, and analyse the BIE (proving continuity and coercivity estimates) all in the context of appropriate fractional Sobolev spaces. This requires new ideas compared to previous work for closed surfaces \citep{HeLaMe:11,NonConvex}.

An outline of the paper is as follows:  we begin in~\S\ref{sec:sobolev} by reviewing details of the Sobolev spaces in which our analysis holds. In~\S\ref{sec:prob} we reformulate Problems $\sP$ and $\sP'$ as a BIE on $\Gamma$, namely a first-kind equation involving the single-layer boundary integral operator $S_k$. We also state $k$-explicit continuity and coercivity estimates for $S_k$ (first stated in~\cite{HeCh:13} and recently proved in~\cite{ScreenCoercivity,ScreenCoercivityPaper} for a much more general class of three-dimensional screens) which are vital for our numerical analysis.  Regularity results for the solution of the BIE are stated in~\S\ref{sec:reg}, where we show how to express the solution as a sum of products of (known) oscillatory functions with (unknown) non-oscillatory amplitudes, for which we have precise regularity estimates. Deriving these estimates requires us to establish a bound on the supremum of $u$ over the whole propagation domain $D$; because of the strong edge singularity this is considerably more complicated than the analogous calculations for convex polygons in \cite{HeLaMe:11}, with separate bounds required close to and away from the screen. 
The results of~\S\ref{sec:reg} are used in~\S\ref{ApproxSpace} to design our $hp$ HNA approximation space, for which we prove rigorous best approximation estimates showing that the number of degrees of freedom required to achieve any prescribed level of accuracy grows only logarithmically with respect to $k$ as $k\rightarrow\infty$.  In~\S\ref{sec:gal} we describe the Galerkin BEM, and derive error estimates for the Galerkin solution, the $k$-dependence of which closely mimics that of the best approximation estimates.  
Numerical results supporting our theory are provided in~\S\ref{sec:num}; these demonstrate that in practice the computational cost required to achieve a fixed accuracy is essentially independent of the wavenumber $k$.  Implementation details and further numerical results (for single scatterers, i.e.\ $n_i=1$) can be found in~\cite{Tw:13}, and related algorithmic ideas for three-dimensional screens can be found in \citet[\S7.6]{ChGrLaSp:11}. 
We remark that we believe that essentially the same numerical method proposed here could be extended to different boundary conditions, e.g.\ Neumann, impedance, as could much of the analysis (in particular the regularity and best approximation results), but we leave this to future work.

Given the wide range of applications of the problem, and the surprising apparent lack of cross-fertilization of ideas in this area between the acoustics, electromagnetics and water waves communities, we feel it beneficial to 
conclude this introductory section with a brief review of alternative analytical, numerical and asymptotic methods proposed in the literature to date.  

In the case $n_i=1$, both Problems $\sP$ and $\sP'$ can be solved via separation of variables in elliptical coordinates, viewing the screen as a degenerate ellipse.  This allows the representation of the solution $u$ as an infinite series of Mathieu functions (see, e.g., \citet[Chapter~4]{Bowman}). However, the series is not straightforward to evaluate in practice, particularly when $k$ is large, even after an application of Watson's transformation (\cite{Ur:68,Se:11,Li:13}). 
There have been many other attempts to derive exact representations for the solution of this and related problems (see, e.g., \cite{To:89a,To:89,Se:11} and in particular the review article~\cite{LiWo:05}), but to the best of our knowledge none are readily computable across the frequency spectrum.

It is also possible to construct an exact solution for the case where the screen consists of an infinite array of identical evenly spaced components (see, e.g., \cite{DaMa:90,AbWi:97}), but, other than for this very specific case, no such formula is known for the case $n_i>1$.  Much effort has gone into the development of embedding schemes that represent the solution for an arbitrary incident angle in terms of the solution to a small number of problems for specific incident angles (see, e.g., \cite{BiPoSt:00,Shanin:03}), but these approaches still require a solution to those specific problems. Thus, in general, both Problems $\sP$ and $\sP'$ must be solved either numerically, or else asymptotically  in the high ($k\rightarrow\infty$) or low ($k\rightarrow0$) frequency limit.

Asymptotic and numerical approaches are usually viewed as being rather complementary. High frequency asymptotic approaches (see, e.g., \cite{Wo:70, GoBeDa:97, HaTh:03})  
are not error-controllable for fixed $k$, but their accuracy improves as $k$ increases.  
By contrast, standard (piecewise polynomial) numerical schemes (see, e.g., \citet{StWe:84,StSu:89}) are error-controllable for fixed $k$, but their computational cost grows at least linearly with respect to $k$ as $k$ increases. (Acceleration techniques such as the fast multipole method may make ``brute-force'' numerical calculations feasible at relatively large $k$, but they do not change the linear growth in computational cost as $k$ increases.) 
This issue is well documented - see, e.g., \citet{ChGrLaSp:11} and the many references therein. 
But the message of the current paper is that, by carefully hybridising the two approaches, one can design numerical methods which perform well across the whole frequency range. 

To the best of our knowledge, the only other numerical scheme for the screen problem that shows anything approaching similar performance (in terms of $k$-dependence) to that achieved here is that proposed by \citet{DavisChew08}.  They exploit the same decomposition as us (motivated by the Geometrical Theory of Diffraction), combined with a coordinate transform to concentrate mesh nodes near strip edges, to derive a numerical solution that is, in their words, ``error controllable and exhibits a bounded error over the full range of frequencies''.  The results given in~\cite{DavisChew08}, for a single strip of widths ranging from half a wavelength up to 1000 wavelengths, are indeed impressive, but the method is not supported by analysis.  We must also mention schemes proposed by \citet{Nye:02} and \citet{ShVa:11}.  A numerical scheme is proposed in \citet[\S11]{Nye:02} for which, by judiciously subtracting and then adding the geometrical optics solution, an improvement in accuracy can be achieved at high frequencies; in this case however the dependence of the accuracy and computational cost on the frequency and discretisation parameters is not made clear.  \citet{ShVa:11} present a novel method (without analysis) that combines a numerical scheme with an asymptotic series; they claim that their scheme is more efficient than a (standard) boundary integral equation method at high frequencies, but the accuracy does appear to deteriorate as frequency increases.

In a recent series of papers \citep{BrLi:12,BrLi:13,LiBr:13}, Bruno and Lintner describe a new framework for problems of scattering by open surfaces, including the sound-soft and sound-hard screen.  They introduce new weighted integral operators, and derive second kind integral equations for both the sound-soft and sound-hard problems (as opposed to the first kind integral equation we solve here).  The enhanced regularity allows the application of high order quadrature rules, and also the use of efficient iterative solvers, leading to significantly reduced computational cost compared to more classical formulations such as that considered here.  In particular, this allows the solution of problems over a wider range of frequencies (in 2D and 3D) than would be possible with more classical approaches, this improvement being particularly noticeable for the sound-hard case.  We note however that, although the algorithms described in \citet{BrLi:12, BrLi:13} are very efficient, the computational cost still grows rapidly as frequency increases, compared to the frequency-independent computational cost that we see in our scheme. 

In future work, it might be of interest to use elements of our approximation space (defined in~\S\ref{ApproxSpace}) with the weighted integral operators proposed by \citet{BrLi:12,LiBr:13}, to see what gains in efficiency might be possible.  But generalising our analysis would require some extra work, because while our best approximation results are independent of the integral equation formulation, the fact that the Galerkin BEM achieves a quasi-optimal approximation (cf.\ \rf{eqn:quasi-opt}) is not. 
\section{Sobolev spaces}
\label{sec:sobolev}
Our analysis is in the context of Sobolev spaces $H^{s} (\Gamma)$ and $\tilde{H}^s(\Gamma)$ for $s\in\R$. We set out here our notation and the basic definitions; for more detail see \citet[\S2]{ScreenCoercivity}.  For $s\in \R$ and integer $n\geq 1$ we define $H^s(\R^n)$ to be the space of those tempered distributions $u$ on $\R^n$ whose Fourier transform satisfies 
\mbox{$\int_{\R^n}(1+|\bxi|^2)^{s}\,|\hat{u}(\bxi)|^2\,\rd \bxi < \infty.$}
Our convention regarding the Fourier transform is that, for $u\in C^\infty_0(\R^n)$,
\mbox{$\hat{u}(\bxi) := (2\pi)^{-n/2}\int_{\R^n}\re^{-\ri \bxi\cdot \bx}u(\bx)\,\rd \bx$}, for $\bxi\in\R^n$. 
In line with many other analyses of high frequency scattering, e.g., \cite{Ih:98}, %
we work with wavenumber-dependent norms. Specifically, we use the norm on $H^s(\R^n)$ defined by
\begin{align}
\norm{u}{H_k^{s}(\R^n)}^2 :=\int_{\R^n}(k^2+|\bxi|^2)^{s}\,|\hat{u}(\bxi)|^2\,\rd \bxi.
\end{align}
We emphasize that $\norm{\cdot}{H^{s}(\R^n)}:= \norm{\cdot}{H_1^{s}(\R^n)}$ is the standard norm on $H^s(\R^n)$, but that, for $k>0$, $\norm{\cdot}{H_k^{s}(\R^n)}$ is another, equivalent, norm on $H^s(\R^n)$. Explicitly,%
\[
  \min\{1,k^s\}\norm{u}{H^{s}(\R^n)} \leq \norm{u}{H^{s}_k(\R^n)} \leq \max\{1,k^s\}  \norm{u}{H^{s}(\R^n)}, \quad \mbox{for }u \in H^s(\R^n).
\]
It is standard that $C_0^\infty(\R^n)$ %
is a dense subset of $H^s(\R^n)$. 
It is also standard (see, e.g.,~\cite{McLean}) that $H^{-s}(\R^n)$ is a natural isometric realisation of $(H^s(\R^n))^*$, the dual space of bounded antilinear functionals on $H^s(\R^n)$, in the sense that the mapping $u\mapsto u^*$  from $H^{-s}(\R^n)$ to $(H^s(\R^n))^*$, defined by
\begin{align}
\label{DualDef}
u^*(v) := \left\langle u, v \right\rangle_{H^{-s}(\R^n)\times H^{s}(\R^n)}:= \int_{\R^n}\hat{u}(\bxi) \overline{\hat{v}(\bxi)}\,\rd \bxi, \quad v\in H^s(\R^n),
\end{align}
is an isometric isomorphism. The duality pairing $\left\langle \cdot, \cdot \right\rangle_{H^{-s}(\R^n)\times H^{s}(\R^n)}$ defined in \rf{DualDef} represents the natural extension of the $L^2(\R^n)$ inner product in the sense that if $u_j, v_j\in L^2(\R^n)$ for each $j$ and $u_j\to u$ and $v_j\to v$ as $j\to\infty$, with respect to the norms on $H^{-s}(\R^n)$ and $H^{s}(\R^n)$ respectively, then $\left\langle u, v \right\rangle_{H^{-s}(\R^n)\times H^{s}(\R^n)}=\lim_{j\to\infty}\left(u_j, v_j \right)_{L^2(\R^n)}$.

We define two Sobolev spaces on $\Omega$ when $\Omega$ is a non-empty open subset of $\R^n$. First, let
$$
H^s(\Omega):=\{U|_\Omega:U\in H^s(\R^n)\},
$$
where $U|_\Omega$ denotes the restriction of the distribution $U$ to $\Omega$ (cf.\ e.g.\ %
\citet[p.~66]{McLean}), with norm%
\begin{align*}
\|u\|_{H_k^{s}(\Omega)}:=\inf_{U\in H^s(\R^n), \,U|_{\Omega}=u}\normt{U}{H_k^{s}(\R^n)}.
\end{align*}
Then $C_{\rm comp}^\infty(\Omega):=\{U|_\Omega:U\in C_0^\infty(\R^n) \}$ is a dense subset of $H^s(\Omega)$. 
Second, let $\tilde{H}^s(\Omega)$ denote the closure of $C_0^\infty(\Omega):=\{U\in C_0^\infty(\R^n): \supp(U)\subset \Omega \}$ in the space $H^s(\R^n)$, equipped with the norm %
$\|\cdot\|_{\tilde H^s_k(\Omega)}$:= $\|\cdot\|_{H^s_k(\R^n)}$. 
When $\Omega$ is sufficiently regular (e.g.\ when $\Omega$ is $C^0$, cf.\ \citet[Thm 3.29]{McLean}) %
we have that $\tilde H^s(\Omega)=\{u\in H^s(\R^n):\supp{u}\subset\overline{\Omega}\}$.

For $s\in \R$ and $\Omega$ any open, non-empty subset of $\R^n$ it holds that
\begin{align}
\label{isdual}
H^{-s}(\Omega)=(\tilde{H}^s(\Omega))^* \; \mbox{ and }\; \tilde H^{s}(\Omega)=(H^{-s}(\Omega))^*,
\end{align}
in the sense that the natural embeddings
$\mathcal{I}:H^{-s}(\Omega)\to(\tilde{H}^s(\Omega))^*$ 
and 
$\mathcal{I}^*:\tilde{H}^{s}(\Omega)\to(H^{-s}(\Omega))^*$,
\begin{align*}
(\mathcal{I} u)(v)
& := \langle u,v \rangle_{H^{-s}(\Omega)\times \tilde{H}^s(\Omega)}:=\langle U,v \rangle_{H^{-s}(\R^n)\times H^s(\R^n)},\\
 (\mathcal{I}^*v)(u)
&:= \langle v,u \rangle_{\tilde H^{s}(\Omega)\times H^{-s}(\Omega)}:=\langle v,U \rangle_{H^{s}(\R^n)\times H^{-s}(\R^n)},
\end{align*}
where $U\in H^{-s}(\R^n)$ is any extension of $u\in H^{-s}(\Omega)$ with $U|_\Omega=u$, are unitary isomorphisms. Also, %
\begin{align}
\label{}
| \langle u,v \rangle_{H^{-s}(\Omega)\times \tilde{H}^s(\Omega)}| =| \langle v,u \rangle_{\tilde{H}^{s}(\Omega)\times H^{-s}(\Omega)}| \leq \norm{u}{H^{-s}_k(\Omega)}\norm{v}{\tilde{H}^{s}_k(\Omega)}, \quad u\in H^{-s}(\Omega), v\in  \tilde{H}^s(\Omega).
\end{align}
We remark that the representations \eqref{isdual} for the dual spaces are well known when $\Omega$ is sufficiently regular. However, it does not appear to be widely appreciated, at least in the numerical analysis for partial differential equations community, that \eqref{isdual} holds without any constraint on the geometry of~$\Omega$. A proof of this general result has been provided recently in \citet[Thm 2.1]{ScreenCoercivity}.

Sobolev spaces can also be defined, for $s\geq 0$, as subspaces of $L^2(\R^n)$ satisfying constraints on weak derivatives. In particular, given a non-empty open subset $\Omega$ of $\R^n$, let
$$
W^1(\Omega) := \{u\in L^2(\Omega): \nabla u \in L^2(\Omega)\},
$$
where $\nabla u$ is the weak gradient. Note that $W^1(\R^n)=H^1(\R^n)$ with
$$
\|u\|^2_{H^1_k(\R^n)} = \int_{\R^n} \left( |\nabla u(\bx)|^2 + k^2 |u(\bx)|^2\right) \rd \bx.
$$
 Further \cite[Theorem 3.30]{McLean}, $W^1(\Omega)= H^1(\Omega)$ whenever $\Omega$ is a Lipschitz open set, in the sense of, e.g., \cite{sauter-schwab11, ChGrLaSp:11}.   It is convenient to define %
$$
W^1_{\mathrm{loc}}(\Omega) := \{u\in L^2_{\mathrm{loc}}(\Omega): \nabla u \in L^2_{\mathrm{loc}}(\Omega)\},
$$
where $L^2_{\mathrm{loc}}(\Omega)$ denotes the set of locally integrable functions $u$ on $\Omega$ for which $\int_G|u(\bx)|^2 \rd \bx < \infty$ for every bounded measurable $G\subset \Omega$.

To define Sobolev spaces on the screen/aperture $\Gamma$ defined by \rf{eqn:GammaDef} we make the natural associations of $\Gamma_\infty$ with $\R$ and of $\Gamma\subset\Gamma_\infty$ with $\tGamma\subset \R$ and set 
$H^s(\Gamma_\infty) := H^s(\R)$, $H^{s}(\Gamma):=H^{s}(\tGamma)$ and $\tilde{H}^{s}(\Gamma):=\tilde{H}^{s}(\tGamma)$, $C^\infty_0(\Gamma_\infty):=C^\infty_0(\R)$, $C^\infty_0(\Gamma):=C^\infty_0(\tGamma)$ etc. 
Recalling that $U^+$ and $U^-$ denote the upper and lower half-planes, respectively, %
we define trace operators  $\gamma^\pm:C_{\mathrm{comp}}^\infty(U^\pm)\to C_{0}^\infty(\Gamma_\infty)$ by $\gamma^\pm u := u|_{\Gamma_\infty}$. It is well known that these extend to bounded linear operators $\gamma^\pm:W^1(U^\pm)\to H^{1/2}(\Gamma_\infty)$. 
Similarly, we define normal derivative operators $\partial_{\bn}^\pm:C_{\mathrm{comp}}^\infty(U^\pm)\to C_{0}^\infty(\Gamma_\infty)$ by 
$\partial_{\bn}^\pm u = \pdonetext{u}{x_2}|_{\Gamma_\infty}$ (so the normal points into $U^+$), which 
extend  (see, e.g.,\ \cite{ChGrLaSp:11}) to bounded linear operators $\partial_\bn^\pm:W^1(U^\pm;\Delta)\to H^{-1/2}(\Gamma_\infty)=(H^{1/2}(\Gamma_\infty))^*$, where $W^1(U^\pm;\Delta):= \{u\in H^1(U^\pm):\Delta u\in L^2(U^\pm)\}$ and $\Delta u$ is the (weak) Laplacian.  Finally, we denote the duality pairing on $H^{1/2}(\Gamma)\times\tilde{H}^{-1/2}(\Gamma)$ by $\langle \cdot,\cdot \rangle_\Gamma$.

\section{Integral equation formulation}
\label{sec:prob}

We now consider the reformulation of the BVPs (Problems $\sP$ and $\sP'$) as integral equations on $\Gamma$; for more detail see \citet[\S3 and \S8]{ScreenCoercivity}. We define the single-layer potential $\cS_k:\tilde{H}^{-1/2}(\Gamma)\to C^2(D)\cap W^1_{loc}(\R^2)$ by 
\begin{align*}
\label{}
\cS_k\phi (\bx)&:=\left\langle \Phi_k(\bx,\cdot),\overline{\phi}\right\rangle_{\Gamma}, \qquad \bx\in \R^2,%
\end{align*}
where $\Phi_k(\bx,\by) := (\ri/4)H_0^{(1)}(k|\bx-\by|)$. %
For $\phi\in L^p(\Gamma)$, with $p>1$, %
the following integral representation holds:
\begin{align*}
\cS_k\phi (\bx) &= \int_\Gamma \Phi_k(\bx,\by) \phi(\by)\, \rd s(\by), \qquad \bx\in \R^2.
\end{align*}
We also define the single-layer boundary integral operator $S_k:\tilde{H}^{-1/2}(\Gamma)\to H^{1/2}(\Gamma)$ by
\begin{align*}
\label{}
S_k\phi &:=\gamma^\pm(\chi\cS_k\phi)|_\Gamma,%
\end{align*}
where $\chi$ is any element of 
$C^\infty_{0,1}(\R^2):=\{\phi \in C_0^\infty(\R^2)$: $\phi=1$ in some neighbourhood of $\Gamma\}$, 
and either of the $\pm$ traces may be taken. For $\phi\in L^p(\Gamma)$, with $p>1$, %
\begin{align}
\label{SDef}
S_k\phi (\bx) &= \int_\Gamma \Phi_k(\bx,\by) \phi(\by)\, \rd s(\by), \qquad \bx\in \Gamma.
\end{align}
Problems $\sP$ and $\sP'$ are equivalent to the same integral equation involving $S_k$, as is made clear by the following theorems, which follow from \citet[Thms~3.8 and~8.6]{ScreenCoercivity} (see also \citet[Theorem~1.7]{StWe:84}).
\begin{thm}
\label{DirEquivThm}
Suppose that $u$ is a solution of Problem $\sP$. Then the representation formula 
\begin{align}
\label{eqn:D_RepThm}
u(\bx )=u^i(\bx)  -\cS_k\left[\pdonetext{u}{\bn}\right](\bx), \qquad\bx\in D,
\end{align}
holds, where 
$[\pdonetext{u}{\bn}]:=\partial^+_\bn(\chi u)-\partial^-_\bn(\chi u)\in \tilde{H}^{-1/2}(\Gamma)$, 
and $\chi$ is an arbitrary element of $C^\infty_{0,1}(\R^2)$. Furthermore, $\phi:=[\pdonetext{u}{\bn}]\in \tilde{H}^{-1/2}(\Gamma)$ satisfies the integral equation
\begin{align}
\label{BIE_sl}
S_k \phi=f,
\end{align}
where $f:=u^i|_\Gamma\in H^{1/2}(\Gamma)$. 
Conversely, suppose that $\phi\in \tilde{H}^{-1/2}(\Gamma)$ satisfies \rf{BIE_sl}. 
Then $u:=u^i-\cS_k\phi$ satisfies Problem $\sP$, and $[\pdonetext{u}{\bn}]=\phi$.
\end{thm}

\begin{thm}
\label{NeuScreenEquivThm}
Suppose that $u'$ is a solution of Problem $\sP'$. Then the representation formula 
\begin{align}
\label{eqn:Dp_RepThm}
u'(\bx )=
\begin{cases}
u^i(\bx) + u^r(\bx) -\cS_k\sumpm{\pdonetext{u'}{\bn}}(\bx), &\bx\in U^+,\\
\cS_k\sumpm{\pdonetext{u'}{\bn}}(\bx), &\bx\in U^-,
\end{cases}
\end{align}
holds, where 
$\sumpm{\pdonetext{u'}{\bn}}(\bx):=\partial^+_\bn(\chi u')+\partial^-_\bn(\chi u')\in \tilde{H}^{-1/2}(\Gamma)$, and $\chi$ is an arbitrary element of $C^\infty_{0,1}(\R^2)$. Furthermore, $\sumpm{\pdonetext{u'}{\bn}}(\bx)\in \tilde{H}^{-1/2}(\Gamma)$ satisfies the integral equation \rf{BIE_sl}. 
Conversely, suppose that $\phi\in \tilde{H}^{-1/2}(\Gamma)$ satisfies \rf{BIE_sl}. Then $u'$, defined by $u':=u^i + u^r-\cS_k\phi$ in $U^+$ and $u':=\cS_k\phi$ in $U^-$, satisfies Problem $\sP'$, and $\sumpm{\pdonetext{u'}{\bn}}=\phi$.
\end{thm}

The following continuity and coercivity properties of the operator $S_k$ have been proved recently in \citet{ScreenCoercivity,ScreenCoercivityPaper}:
\begin{lem}[{\citet[Theorem~5.2]{ScreenCoercivity}}]
\label{ThmSNormSmooth}
Let $s\in \R$. Then $S_k:\tilde{H}^{s}(\Gamma)\to H^{s+1}(\Gamma)$ is bounded, and for $kL\geq c_0>0$ there exists a constant $C_0>0$, depending only on $c_0$ (specifically, $C_0=C\log(2+c_0^{-1})$, where $C$ is independent of $c_0$),  %
such that
\begin{align}
\label{SEst2D}
\norm{S_k\phi}{H^{s+1}_k(\Gamma)} \leq C_0 (1+\sqrt{kL})\norm{\phi}{\tilde{H}^{s}_k(\Gamma)}, \qquad 
\phi\in \tilde{H}^{s}(\Gamma).
\end{align}
\end{lem}
\begin{lem}[{\citet[Theorem~5.3]{ScreenCoercivity}}]
\label{ThmCoerc}
$S_k:\tilde{H}^{-1/2}(\Gamma)\to H^{1/2}(\Gamma)$ satisfies
\begin{align}
\label{SCoercive}
\left |\langle S_k \phi,\phi\rangle_{\Gamma}\right|
\geq \frac{1}{2\sqrt{2}} \norm{\phi}{\tilde{H}^{-1/2}_k(\Gamma)}^2, \qquad k>0, \,\,\, \phi\in \tilde{H}^{-1/2}(\Gamma).
\end{align}
\end{lem}
These results, combined with the standard Lax-Milgram Lemma, 
imply the unique solvability in $\tilde{H}^{-1/2}(\Gamma)$ of the integral equation \rf{BIE_sl} for all $k>0$. In particular we obtain the stability estimate
\begin{align}
\label{eqn:StabEst}
\normt{S_k^{-1}\psi}{\tilde{H}_k^{-1/2}(\Gamma)}\leq 2\sqrt{2} \norm{\psi}{H_k^{1/2}(\Gamma)}, \qquad \psi\in H^{1/2}(\Gamma).
\end{align}
Moreover, Theorems \ref{DirEquivThm} and \ref{NeuScreenEquivThm} then imply the unique solvability of the BVPs:
\begin{thm}
\label{DirExUn}
Problems $\sP$ and $\sP'$ each have unique solutions $u$ and $u'$ for all $k>0$, which satisfy %
\begin{equation}
\label{eqn:corres}
\begin{split}
u(\bx )=u'(\bx)- u^r(\bx), \qquad \bx\in U^+,\\
u(\bx )=u^i(\bx) - u'(\bx), \qquad \bx\in U^-.
\end{split}
\end{equation}
\end{thm}

\section{Analyticity and regularity of solutions}
\label{sec:reg}

Standard elliptic regularity results imply that the unique solution of Problem~$\sP$ is continuous up to $\Gamma$, so that $u\in C(\R^2)$. Since $u(\bx)\sim u^i(\bx)$ as $|\bx|\to\infty$, it follows that
\[
  \uM:=\sup_{\bx\in D}|u(\bx)| <\infty.
\]
In fact $u$ is H\"older continuous with index $1/2$. In particular, defining
$$
 \ell_{min}:=\min_{m\in\{1,\ldots,2n_i-1\}}( s_{m+1}-s_m),
$$
if $k\ell_{min}\geq c_0$ for some $c_0>0$ then 
\begin{equation}
  |u(\bx)| \leq C \uM (kd)^{1/2}, \quad \bx\in D,
  \label{eqn:ubound}
\end{equation}
where $d:= \dist(\bx,\Gamma)$ and $C$ is a constant that depends only on $c_0$. 
Since $u=0$ on $\Gamma$ this is clear by reflection arguments and standard interior elliptic regularity results (e.g.\ \citet[Lemma~2.1]{ChZh:98}) except in a neighbourhood of the corners of $\Gamma$. But near these corners the bound~(\ref{eqn:ubound}) follows from the explicit separation of variables representation for the solution, equation (\ref{eqn:u_r_theta}) below (for more detail see the very similar arguments in \citet[Lemma~3.5]{HeLaMe:11}).

Our HNA method for solving~\rf{BIE_sl} (as a means of solving the BVPs $\sP$ and $\sP'$) uses an approximation space (defined explicitly in \S\ref{ApproxSpace}) which is adapted to the high frequency asymptotic behaviour of the unknown $\phi=[\pdonetext{u}{\bn}]=\sumpm{\pdonetext{u'}{\bn}}$, which we now consider. 
We represent the point $\bx\in\Gamma$ parametrically by $\bx(s):=(s,0)$, where $s\in\tGamma\subset[0,L]$.  Combining the bound~(\ref{eqn:ubound}) with elementary bounds on integral representations for $u$ in the upper and lower half-planes, arguing exactly as in the proof of \citet[Theorem~3.2]{HeLaMe:11}), one can prove the following:
\begin{thm}
\label{vpmThm}
Let $k\ell_{min}\geq c_0>0$. Then for any $j=1,\ldots,n_i$, we have the decomposition
\begin{align}
\label{Decomp}
\phi(\bx(s))
= \Psi(\bx(s)) + v_j^+(s-s_{2j-1})\re^{\ri ks} +v_j^-(s_{2j}-s) \re^{-\ri ks}, & \quad s\in(s_{2j-1},s_{2j}), 
\end{align}
where $\Psi := 2\pdonetext{u^i}{\bn}$, and the functions $v_j^\pm(s)$ are analytic in the right half-plane $\real{s}>0$, with %
\begin{align}
\label{vpmBounds}
|v_j^\pm(s)|\leq C_1 \uM k|ks|^{-1/2}, \qquad \real{s}>0,
\end{align}
where the constant $C_1>0$ depends only 
on $c_0$.%
\end{thm}

\begin{rem}
\label{rem:NonOscillatory}
The analyticity of the functions $v_j^\pm$ and the bound \rf{vpmBounds} imply that $v_j^\pm$ are non-oscillatory. Explicitly, by the Cauchy integral formula for derivatives, the derivatives of $v_j^\pm$ satisfy bounds of the form 
$|{v_j^\pm}^{(n)}(s)| \leq c_n C_1 M(u)k^{1/2}s^{-(n+1/2)}$ for $s>0$, $n\in\N_0$ and $c_n$ a constant depending only on $n$. The lack of oscillation is indicated by the fact that the $k$-dependence of these bounds is the same for all $n$.
\end{rem}

\begin{rem}
\label{rem:MMprime}
Note that by the correspondence~\rf{eqn:corres} we can bound $\uM$ above and below by a multiple of $\uMp$, precisely $\uMp/2\leq \uM\leq 2\uMp$.
\end{rem}

\begin{rem}
\label{rem:PO}
For the screen Problem $\sP$, the representation~\rf{Decomp} can be interpreted in terms of high frequency asymptotic theory as follows. The first term, $\Psi$, is the 
geometrical optics (GO) approximation to $\phi=[\pdonetext{u}{\bn}]$, representing the direct contribution of the incident and reflected waves.  (Using this approximation alone in the representation~\rf{eqn:D_RepThm} gives the ``physical optics'' approximation of $u$ in $D$.)  The second and third terms in~\rf{Decomp} represent the combined contribution of all the diffracted waves (including multiply-diffracted waves that have travelled arbitrarily many times along and between the different components of the screen) propagating right (oscillating like $\re^{\ri ks}$ and with a singularity at $s_{2j-1}$) and left (oscillating like $\re^{-\ri ks}$ and with a singularity at $s_{2j}$) respectively along the screen segment $\Gamma_j$. A similar interpretation holds for the aperture Problem $\sP'$.  Comparing~(\ref{vpmBounds}) with \citet[(3.5)]{HeLaMe:11}, we see that for $|s|>1/k$ our functions $v_j^{\pm}$ satisfy an identical bound to the comparable functions for the problem of scattering by convex polygons; for $|s|<1/k$ however, the singularity is stronger for the screen problem, with the exponent of $-1/2$ in~(\ref{vpmBounds}) comparing to an exponent in the interval $(-1/2,0)$ for scattering by convex polygons, with the exact value dependent on the corner angle.  This makes clear the fact, alluded to in \S\ref{Intro}, that $v_j^{\pm}\not\in L^2(\Gamma)$ (unlike the comparable functions for scattering by convex polygons).%
\end{rem}

The dependence of the constant $\uM$ in \rf{vpmBounds} on the wavenumber $k$ is not yet fully understood. The following lemma provides an upper bound on $\uM$ which implies that $\uM=\ord{k}$ as $k\to\infty$. However, we do not believe this bound is sharp; in \citet[Theorem 4.3]{HeLaMe:11} it is shown for the case of scattering by a star-like sound-soft polygon that $\uM = \ord{k^{1/2}\log^{1/2}k}$ as $k\to\infty$, uniformly with respect to the angle of incidence, with numerical results therein suggesting the plausibility of the hypothesis $\uM=\ord{1}$ as $k\to\infty$.  Such a hypothesis is also plausible for the screen problem, and consistent with the numerical results in~\S\ref{sec:num}, but we cannot yet prove this.%

\begin{lem}
\label{lem:Mu}
Let $\Gamma$ be of the form \rf{eqn:GammaDef} and let $k\ell_{min}\geq c_0>0$. Then there exists a constant $C_2>0$, depending only on $c_0$, 
such that
\begin{align*}
 \label{}
 \uM\leq C_2 (1+kL).%
 \end{align*} 
\end{lem}

The remainder of this section consists entirely of the proof of Lemma~\ref{lem:Mu}; readers more interested in the numerical method may skip immediately to~\S\ref{ApproxSpace}.  The proof of Lemma~\ref{lem:Mu} comprises three stages. First, in Proposition~\ref{cor:SolnBound}, we derive a pointwise bound on $|u(\bx)|$ which is valid in the whole domain $D$, but which is non-uniform. %
This bound follows from the following lemma, a proof of which can be found in \citet[Lemma 7.1]{ScreenCoercivity} - see also \citet{ScreenCoercivityPaper} for slightly sharper bounds.%
\begin{lem}
\label{lem:H_one_half_norms}
Let $k>0$ and let $\Gamma$ be of the form \rf{eqn:GammaDef}, with $\Gamma_\infty$, $D$ and $L$ defined as in~\S\ref{Intro}.%
\begin{enumerate}[(i)]
\item Let $\bd\in\R^2$ with $|\bd|\leq 1$. Then there exists $C>0$, independent of $\bd$, $k$ and $\Gamma$, such that 
\begin{align*}
 \label{}
 \normt{\re^{\ri k \bd\cdot (\cdot)}}{H^{1/2}_k(\Gamma)} \leq C(1+\sqrt{kL}).
 \end{align*} 
\item Let $\bx\in D$ and $d:=\dist(\bx,\Gamma)$. %
Then there exists $C>0$, independent of $\bx$, $k$ and $\Gamma$, such that
 \begin{align*}
 \label{}
 \norm{\Phi_k(\bx,\cdot)}{H^{1/2}_k(\Gamma)} \leq C\left(1+\frac{1}{\sqrt{kL}}\right)\left(1+\frac{1}{\sqrt{kd}}\right)\log\left(2+\frac{1}{kd}\right)\log^{1/2}(2+kL).
 \end{align*} 
\end{enumerate}
\end{lem}
From Lemma~\ref{lem:H_one_half_norms}, Theorem~\ref{DirEquivThm} and~\rf{eqn:StabEst}, one can derive the following result (cf.\ \citet[Cor. 7.2]{ScreenCoercivity}). 
Note that the bound~\rf{eqn:pointwisebound} blows up as $\bx$ approaches $\Gamma$ (i.e.\ as $d\to 0$).%
\begin{prop}
\label{cor:SolnBound}
The solution $u$ of Problem $\sP$ satisfies the pointwise bound
\begin{align}
\label{eqn:pointwisebound}
|u(\bx)|\leq C\left(1+\frac{1}{\sqrt{kL}}\right)\left(1+\frac{1}{\sqrt{kd}}\right)\log\left(2+\frac{1}{kd}\right)\log^{1/2}(2+kL)(1+\sqrt{kL}), \qquad \bx\in D,
\end{align}
where $d=\dist(\bx,\Gamma)$, and $C>0$ is independent of $\bx$, $k$ and $\Gamma$. 
\end{prop}

The second stage in the proof of Lemma~\ref{lem:Mu} involves the derivation of a uniform bound on $|u(\bx)|$ valid on a neighbourhood of $\Gamma$. We begin by using a separation of variables argument to bound $|u(\bx)|$ close to $\Gamma$ in terms of the $L^2$ norm of the scattered field in a neighbourhood of $\Gamma$.
\begin{lem}
\label{lem:SepVarsInterval}
Let $\Gamma$ be of the form \rf{eqn:GammaDef}, and let $u$ be the corresponding solution of Problem $\sP$, with
$u^s=u-u^i$. 
Let $\eps_*:=\min\{\ell_{min}/2,\pi/(3k)\}$, where $\ell_{min}$ is defined as at the start of this section. Then
\begin{align}
\label{eqn:SepVarEst}
|u(\bx)|\leq \frac{32}{3\sqrt{\pi}(\sqrt{2}-1)}\left(1+ k\left(1+(k\ell_{min})^{-1}\right)\norm{u^s}{L^2((\Gamma)_{\eps_*})}\right), \qquad \bx\in(\Gamma)_{{\eps_*}/32},
\end{align}
where for $E\subset\R^2$ and $\eps>0$, $(E)_{\eps} := \{\bx\in\R^2: \, \dist(\bx,E)\leq\eps\}$.%
\end{lem}
\begin{proof}
First pick $j\in \{1,\ldots,n_i\}$ and let $\be_j:=(s_{2j-1},0)$ and $\be_j':=(s_{2j},0)$ denote respectively the left and right endpoints of the segment $\Gamma_j$. Let $(r,\theta)$ be polar coordinates centered at $\be_j$, such that $\Gamma_j$ is described by the set $\{(r,\theta): 0<r<L_j, \,\theta=0 \textrm{ or }\theta=2\pi\}$. 
Then for any $0<R<\ell_{min}$ (so that we avoid the singularities at the endpoints of the segments) the restriction of $u$ to $B_R(\be_j)$ (the ball of radius $R$ centred at $\be_j$) can be written, using separation of variables, as%
\begin{align}
\label{eqn:u_r_theta}
u(r,\theta) = \sum_{n=1}^\infty a_n(R) J_{n/2}(kr)\sin{\left(\frac{n\theta}{2}\right)}, \qquad 0<r<R,\,\,0\leq\theta\leq 2\pi,
\end{align}
where
\begin{align*}
\label{}
a_n(R) = \frac{1}{\pi J_{n/2}(kR)}\int_0^{2\pi} u(R,\theta) \sin{\left(\frac{n\theta}{2}\right)}\,\rd \theta.
\end{align*}
For any $0<\tR<R$ we can derive an identical formula to~\rf{eqn:u_r_theta} with $R$ replaced by $\tR$.  Comparing the two formulae on the ball $B_{\tR}(\be_j)\subset B_{R}(\be_j)$, it follows that, in fact, $a_n(\tR)$ takes the same value (which we denote simply by $a_n$) for all $0\leq \tR \leq R$.  To bound $|a_n|$ we then note that
\begin{align*}
\label{}
\frac{3a_nR^2}{8} = \int_{R/2}^R a_n(\tR)\tR\,\rd \tR = \int_{R/2}^R\int_0^{2\pi} \frac{u(\tR,\theta) \sin{(n\theta/2)}}{\pi J_{n/2}(k\tR)} \tR\,\rd \theta \rd \tR,
\end{align*}
and hence
\begin{align*}
\label{}
|a_n| &\leq \frac{8}{3\pi R^2}\sqrt{\int_{R/2}^R\int_0^{2\pi} \frac{|\sin{(n\theta/2)}|^2}{|J_{n/2}(k\tR)|^2} \tR\,\rd \theta \rd \tR}\,\sqrt{\int_{R/2}^R\int_0^{2\pi} |u(\tR,\theta)|^2 \tR\,\rd \theta \rd \tR}%
&= \frac{8K_n}{3\sqrt{\pi} k R^2} \norm{u}{L^2(A_{R/2,R})},
\end{align*}
$K_n:= \sqrt{\int_{kR/2}^{kR} \frac{z\,\rd z}{|J_{n/2}(z)|^2}}$ and $A_{R/2,R}$ is the annulus defined by $A_{R/2,R}:=\{(r,\theta): R/2<r<R, \,0\leq \theta\leq 2\pi \}$.   
To bound $|K_n|$, we note that (cf.,\ e.g.,\ \citet[(3.12)]{Convex})
\begin{align}
\label{JnuEst}
\cos{z} \leq \frac{J_{\nu}(z)\Gamma(1+\nu)}{(z/2)^{\nu}} \leq 1, \qquad 0\leq z\leq \pi/2, \,\,\nu>-1/2.
\end{align}
where $\Gamma(\cdot)$, in (\ref{JnuEst})--(\ref{anbound}), denotes the Gamma function.
Hence if $0<kR\leq\pi/3$ (so that $1/2\leq \cos{z} \leq 1$ for $kR/2\leq z\leq kR$) then
\begin{align}
\label{Knbound}
|K_n|\leq 2^{1+n/2}\Gamma(1+n/2)\sqrt{\int_{kR/2}^{kR} z^{1-n}\,\rd z}
\leq \frac{2^{1+n}\Gamma(1+n/2)}{\sqrt{n}}(kR)^{1-n/2}.
\end{align}
Thus
\begin{align}
\label{anbound}
|a_n| \leq \frac{2^{4+n}\Gamma(1+n/2)(kR)^{-n/2}}{3\sqrt{\pi} R \sqrt{n}} \norm{u}{L^2(A_{R/2,R})},
\end{align}
and, using \rf{JnuEst} again,%
\begin{align*}
\label{}
|a_n J_{n/2}(kr)| \leq \frac{16(2r/R)^{n/2}}{3\sqrt{\pi} R \sqrt{n}} \norm{u}{L^2(A_{R/2,R})}.
\end{align*}
Then, for $\bx\in B_{R/2}(\be_j)$,
\begin{align*}
\label{}
|u(\bx)|=|u(r,\theta)|\leq \frac{16}{3\sqrt{\pi}R}\sum_{n=1}^\infty (2r/R)^{n/2} \norm{u}{L^2(A_{R/2,R})}
=\frac{16}{3\sqrt{\pi}R}\left(\frac{(2r/R)^{1/2}}{1-(2r/R)^{1/2}}\right)\norm{u}{L^2(A_{R/2,R})}.
\end{align*}
In particular, for $\bx\in B_{R/4}(\be_j)$ we have
\begin{align*}
\label{}
|u(\bx)|\leq \frac{16}{3\sqrt{\pi}(\sqrt{2}-1)R}\norm{u}{L^2(A_{R/2,R})}.
\end{align*}
Recalling that $u=u^i+u^s$, and noting that $\normt{u^i}{L^2(A_{R/2,R})}\leq \sqrt{3\pi}R/2$, this implies that
\begin{align}
\label{xNearEst}
|u(\bx)|\leq \frac{8}{\sqrt{3}(\sqrt{2}-1)}+ \frac{16}{3\sqrt{\pi}(\sqrt{2}-1)R}\norm{u^s}{L^2(A_{R/2,R})},
\qquad \bx\in B_{R/4}(\be_j).
\end{align}
To satisfy both $R\leq \pi/(3k)$ and $R<\ell_{min}$, it suffices to set, e.g., $R=R_j:=\min\{\ell_{min}/2,\pi/(3k)\}$. A similar estimate to \rf{xNearEst} can be obtained in a neighbourhood of the right endpoint $\be_j'$.

Now let $\bx_j$ denote an interior point of $\Gamma_j$ and let $(r,\theta)$ be polar coordinates centered at $\bx_j$, so that $\Gamma_j$ is a subset of the lines $\theta=0$ and $\theta=\pi$. By a similar analysis to that presented above, but with the separation of variables carried out only in a half-disk $0\leq \theta\leq \pi$ or $\pi\leq \theta\leq 2\pi$ and $n/2$ replaced by $n$ etc., we can show that, if $0<R\leq \pi/(3k)$ and $R<\min\{|\bx-\be_j|,|\bx-\be_j'|\}$, then
\begin{align}
\label{xAwayEst}
|u(\bx)|\leq \frac{4\sqrt{2}}{\sqrt{3}}+\frac{16}{3\sqrt{\pi}R}\norm{u^s}{L^2(\tilde{A}_{R/2,R})},
\qquad \bx\in B_{R/4}(\bx_j),
\end{align}
where $\tilde{A}_{R/2,R}:=\{(r,\theta): R/2<r<R, \,0\leq \theta\leq \pi \}$ is a semi-annulus centered at $\bx_j$. 

To combine these results we note that if $\min\{|\bx-\be_j|,|\bx-\be_j'|\}>R_j/4$ then we can take $R=R_j/4$ in \rf{xAwayEst}. Then the union of the balls $B_{R_j/16}(\bx_j)$ over all such $\bx_j$, together with the balls $B_{R_j/4}(\be_j)$ and $B_{R_j/4}(\be_j')$, certainly covers a ($R_j/32$)-neighbourhood of $\Gamma_j$. 
Hence we can conclude that
\begin{align}
\label{}
|u(\bx)|\leq \frac{8}{\sqrt{3}(\sqrt{2}-1)}+ \frac{16}{3\sqrt{\pi}(\sqrt{2}-1)R_j}\norm{u^s}{L^2((\Gamma_j)_{R_j})},
\qquad \bx\in(\Gamma_j)_{R_j/32},
\end{align}
from which the result follows, since $1/R_j \leq 2k(1+(k\ell_{min})^{-1})$.
\end{proof}

To use Lemma \ref{lem:SepVarsInterval} we require an estimate of $\norm{u^s}{L^2((\Gamma)_{\eps_*})}$, %
which is provided by the following result:
\begin{lem}
\label{lem:L2Est}
Let $\eps>0$. Then there exists a constant $C>0$, independent of $\eps$, $k$ and $\Gamma$, such that%
\begin{align}
\label{eqn:L2Est}
\norm{\cS_k\phi}{L^2((\Gamma)_{\eps})} 
&\leq C\sqrt{k\eps}(1+k\eps)k^{-1}\log{(2+(kL)^{-1})}(1+ (kL)^{1/2}) \norm{\phi}{\tilde{H}_k^{-1/2}(\Gamma)}, \quad \phi\in\tilde{H}^{-1/2}(\Gamma).
\end{align}
\end{lem}
\begin{proof}
Arguing as in the proof of \citet[Lemma~5.1 and Thm~5.2]{ScreenCoercivity}, one can show that for any $\eps>0$
(see also \citet{ScreenCoercivityPaper} for slightly sharper bounds)
\begin{align*}
\label{}
\norm{\cS_k\phi(\cdot,x_2)}{L^2(\tGamma_\eps)} 
\leq C \log{(2+(kA)^{-1})}(1+ (kA)^{1/2} + (k|x_2|)^{1/2}\log{(2+kA)}) \norm{\phi}{\tilde{H}_k^{-1}(\Gamma)},
\end{align*}
where $A=L+\eps$,  $\tGamma_\eps:=\{x\in\R:\dist{(x,\tGamma)}<\eps\}$, $x_2\in \R$, and $C>0$ is independent of $k$, $\Gamma$ and $\eps$. %
From this one can show that 
\begin{align*}
\label{}
\norm{\cS_k\phi(\cdot,x_2)}{L^2((\tGamma)_{\eps})} 
&\leq C (1+k\eps) \log{(2+(kL)^{-1})}(1+ (kL)^{1/2}) \norm{\phi}{\tilde{H}_k^{-1}(\Gamma)}, \quad |x_2|\leq \eps,%
\end{align*}
where again $C$ is independent of $k$, $\Gamma$ and $\eps$. 
The estimate \rf{eqn:L2Est} then follows from integrating over $x_2\in(-\eps,\eps)$ and noting that $\norm{\phi}{\tilde{H}_k^{-1}(\Gamma)}\leq k^{-1/2}\norm{\phi}{\tilde{H}_k^{-1/2}(\Gamma)}$. 
\end{proof}

Combining Lemmas \ref{lem:SepVarsInterval} and \ref{lem:L2Est} gives:
\begin{prop}
\label{cor:pointwise2}
Under the assumptions of Lemma \ref{lem:SepVarsInterval}, we have
\begin{align}
\label{eqn:pointwisebound2}
|u(\bx)|\leq C\left(1+(k\ell_{min})^{-1}\right)\log{(2+(kL)^{-1})}(1+ kL), \qquad \bx\in(\Gamma)_{{\eps}_*/32},
\end{align}
where $C>0$ is independent of $\bx$, $k$ and $\Gamma$.
\end{prop}
\begin{proof}
Noting that $u^s = -\cS_k \left[\pdonetext{u}{\bn}\right]$, and that $\left[\pdonetext{u}{\bn}\right] =S_k^{-1}u^i|_\Gamma$, the result follows from Lemmas \ref{lem:H_one_half_norms}(i), \ref{lem:SepVarsInterval} and \ref{lem:L2Est}, the stability estimate \rf{eqn:StabEst}, and the fact that $k\eps_*\leq \pi/3$.
\end{proof}

The third and final stage in the proof of Lemma \ref{lem:Mu} involves combining Propositions \ref{cor:SolnBound} and \ref{cor:pointwise2} to obtain a bound which holds uniformly throughout $D$. Specifically, we combine \rf{eqn:pointwisebound2}, which holds in the region $d<\eps_*/32$, with \rf{eqn:pointwisebound}, applied in the region $d\geq \eps_*/32$. Noting that in the latter case we have $(kd)^{-1}\leq 32/(k\eps_*)\leq C(1+(k\ell_{min})^{-1})$, we can obtain the following estimate in which the constant $C$ is independent of both $k$ and $\Gamma$: 
\begin{align*}
\label{}
|u(\bx)|\leq C\left(1+\frac{1}{k\ell_{min}}\right)\log \left(1+\frac{1}{k\ell_{min}}\right)(1+kL), \qquad \bx\in D.
\end{align*}
The statement of Lemma \ref{lem:Mu} then follows immediately.

\section{$hp$ approximation space and best approximation results}
\label{ApproxSpace}
Our numerical method for solving the integral equation \rf{BIE_sl} uses a hybrid numerical-asymptotic approximation space based on Theorem~\ref{vpmThm}. Rather than approximating $\phi$ itself using piecewise polynomials (as in conventional methods), we use the decomposition~\rf{Decomp}, with the factors $v_j^+$ and $v_j^-$ replaced by piecewise polynomials. The advantage of our approach is that, as is quantified by Theorem \ref{vpmThm}, the functions $v_j^\pm$ are non-oscillatory (cf.\ Remark \ref{rem:NonOscillatory}), and can therefore be approximated much more efficiently than the full (oscillatory) solution $\phi$. 
Explicitly, the function we seek to approximate is
\begin{equation}
  \varphi(s) := \frac{1}{k}\left(\phi(\bx(s))-\Psi(\bx(s))\right), \quad s\in\tGamma\subset(0,L),
  \label{eqn:varphi0}
\end{equation}
which represents the difference between $\phi$ and its GO approximation~$\Psi$ (recall Remark~\ref{rem:PO}), scaled by $1/k$ so that $\varphi$ is nondimensional (cf.~\cite{Convex}). By~\rf{Decomp} we know that
\begin{align}
\label{eqn:varphi}
\varphi(s) = \frac{1}{k} \left( v_j^+(s-s_{2j-1})\re^{\ri ks} +v_j^-(s_{2j}-s) \re^{-\ri ks} \right), \qquad s\in(s_{2j-1},s_{2j}), 
\,\,j=1,\ldots,n_i,
\end{align}
with the factors $v_j^{\pm}$ enjoying the analyticity properties described in Theorem~\ref{vpmThm}. 
Our hybrid approximation space represents $\varphi$ on each segment $\Gamma_j$ in the form \rf{eqn:varphi}, with the factors $v_j^+$ and $v_j^-$ replaced by piecewise polynomials on overlapping meshes, graded towards the singularities at $s=s_{2j-1}$ and $s=s_{2j}$ respectively. For an illustration of the resulting mesh structure on $\Gamma$ see Figure~\ref{fig:grading}. 
We denote our approximation space by $V_{N,k}\subset \tilde{H}^{-1/2}\left(\Gamma\right)$, where $N$ denotes the total number of degrees of freedom in the method (to be elucidated later), and the subscript $k$ serves to indicate that our hybrid approximation space depends explicitly on the wavenumber~$k$. 

To describe in more detail the meshes we use, we consider the case of a geometric mesh on the interval $[0,l]$, $l>0$, refined towards $0$. The meshes for approximating $v_j^\pm$ on each segment $\Gamma_j$ are constructed from this basic building block by straightforward coordinate transformations. 
Given $n\geq 1$ (the number of layers in the mesh) let $G_n(0,l)$ denote the set of meshpoints $\{x_i\}_{i=0}^n$ defined by
\begin{equation}
      x_0:=0,\quad x_i:=\sigma^{n-i}l,\quad i=1,2,\ldots,n,  
      \label{mesh_graded}
\end{equation}
where $0<\sigma <1$ is a grading parameter. A smaller value of $\sigma$ represents a more severe grading - in all of our experiments we take $\sigma=0.15$, as in \cite{HeLaMe:11}. 
Given a vector $\mathbf{p}\in(\mathbb{N}_0)^n$, let $P_{\mathbf{p},n}(0,l)$ denote the space of piecewise polynomials on the mesh $G_n(0,l)$ with the degree vector $\mathbf{p}$, i.e., 
\[
   P_{\mathbf{p},n}(0,l):=\left\{\rho:[0,l]\rightarrow\mathbb{C}:\rho|_{(x_{i-1},x_i)} \text{ is a polynomial of degree less than or equal to } (\mathbf{p})_i,\, i=1,\ldots,n\right\}.
\]
For reasons of efficiency and conditioning it is common to decrease the order of the approximating polynomials towards the singularity. Specifically, we shall consider degree vectors $\bp$ of the form
\begin{align}
\label{piDef1}
(\bs{p})_i:=
\begin{cases}
p - \left\lfloor \frac{\alpha(n+1-i)}{n}p \right\rfloor,& 1\leq i\leq n-1,\\
p, & i=n,
\end{cases}
\end{align}
for some $\alpha \in [0,1]$ and some integer $p\geq 0$ (the highest polynomial degree on the mesh). The choice $\alpha=0$ corresponds to a constant degree across the mesh (this was the only choice considered in~\cite{HeLaMe:11}), while for $\alpha\in(0,1]$ the degree decreases linearly in the direction of refinement.

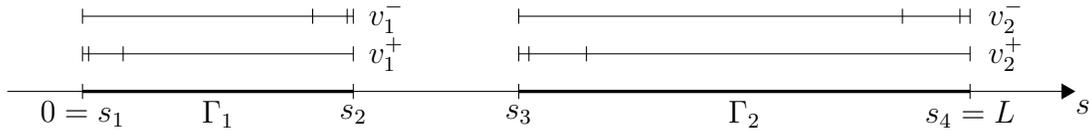
\begin{figure}[t!]
\centering
\begin{tikzpicture}
    \draw[line width=1.25pt] (3,0)--(-3,0);
    \draw (3,0.5)--(-3,0.5);
    \draw (3,1)--(-3,1);
    \foreach \x in {-3,-2.865,-2.1,3}
    {
    \draw (\x ,0.4) -- (\x ,0.6);
    \draw (-\x ,0.9) --(-\x,1.1);
    }
    \path (3.1,1)node[anchor=west]{$v_2^-$};
    \path (3.1,0.5)node[anchor=west]{$v_2^+$};
    \path (0,-0.3)node{$\Gamma_2$};
\begin{scope}[yshift=0.7cm]
    \draw (-3 ,-0.6) -- (-3 ,-0.8);
    \draw (3 ,-0.6) -- (3 ,-0.8);
    \path (-3,-1.0)node{$s_3$};
    \path (3,-1.0)node{$s_4=L$};
\end{scope}
\begin{scope}[x=0.6cm,y=1.0cm,xshift=-7cm]
    \draw[line width=1.25pt] (3,0)--(-3,0);
    \draw (3,0.5)--(-3,0.5);
    \draw (3,1)--(-3,1);
    \foreach \x in {-3,-2.865,-2.1,3}
    {
    \draw (\x ,0.4) -- (\x ,0.6);
    \draw (-\x ,0.9) --(-\x,1.1);
    }
    \path (3.1,1)node[anchor=west]{$v_1^-$};
    \path (3.1,0.5)node[anchor=west]{$v_1^+$};
    \path (0,-0.3)node{$\Gamma_1$};
\begin{scope}[yshift=0.7cm]
    \draw (-3 ,-0.6) -- (-3 ,-0.8);
    \draw (3 ,-0.6) -- (3 ,-0.8);
    \path (-3,-1.0)node{$0=s_1$};
    \path (3,-1.0)node{$s_2$};
\end{scope}
\end{scope}
\begin{scope}[yshift=0.7cm]
    \draw[arrows={-triangle 60}] (-9.8,-0.7)--(4.4,-0.7);
    \path (4.5,-0.9)node{$s$};
\end{scope}
\end{tikzpicture}
\caption{Illustration of the overlapping geometrically graded meshes used to approximate the amplitudes $v_j^\pm$ in \rf{eqn:varphi}, in the case where $\Gamma$ comprises two components, $\Gamma_1$ and $\Gamma_2$. %
}
\label{fig:grading}
\end{figure}

For each $j=1,\ldots,n_i$ let $n_j^\pm\geq 1$ and $\bs{p}_j^\pm\in (\N_0)^{n_j^\pm}$ denote respectively the number of layers and the degree vector associated with the approximation of the factor $v_j^\pm$ in \rf{eqn:varphi}. The total number of degrees of freedom in $V_{N,k}$ is then
\begin{align}
N := \dim(V_{N,k}) =  \sum_{j=1}^{n_i} \left(\sum_{m=1}^{n_j^+}\left((\bs{p}_j^+)_m+1\right)+\sum_{m=1}^{n_j^-}\left((\bs{p}_j^-)_m+1\right) \right).
\label{eqn:dof1}
\end{align}

The regularity results provided by Theorem~\ref{vpmThm} allow us to prove that, under certain assumptions, the best approximation error in approximating $\varphi$ by an element of $V_{N,k}$ decays exponentially as $p$, the maximum degree of the approximating polynomials, increases. 
Our best approximation results in the space $\tilde{H}^{-1/2}(\tGamma)$ are stated in the following theorem, which is the main result of this section.  For simplicity of presentation we assume that the mesh parameters are the same in each of the meshes used to approximate the different components $v_j^\pm$ (similar estimates hold in the more general case).
\begin{thm}
\label{dudnThm}
Let $k\ell_{min}\geq c_0>0$. 
Suppose that $n_j^\pm=n$ and $\bp_j^\pm=\bp$ for each $j=1,\ldots,n_i$, where $n$ and $\bp$ are defined by 
\rf{piDef1} with $n\geq cp$ for some constant $c>0$. 
Then, for any $0<\epsilon<1/2$, there exists a constant $C_3>0$, depending only on $\epsilon$, $\sigma$, $n_i$ and $c_0$, and a constant $\tau>0$, depending only on $\epsilon$, $\sigma$, $\alpha$ and $c$, such that
\begin{align}
\label{BestAppdudn}
\inf_{v\in V_{N,k}}\norm{\varphi-v}{\tilde{H}_k^{-1/2}(\Gamma)}\leq C_3\uM k^{-1}(kL)^\epsilon\,\re^{-p\tau}.%
\end{align}
\end{thm}

The proof of Theorem~\ref{dudnThm} occupies the rest of this section.  It relies on a number of intermediate results, which we now state.  The first of these is a standard application of the Riesz-Thorin interpolation theorem \citep[Chapter~V, Theorem~1.3]{StWe:71}.

\begin{lem}%
\label{FTLpThm}
For $1\leq q\leq 2$ the Fourier transform extends uniquely from $L^2(\R)\cap L^1(\R)$ to a bounded linear operator from $L^q(\R)$ to $L^r(\R)$, where $1/q+1/r=1$ (with $r=\infty$ if $q=1$). Furthermore, with $\theta := 2/q-1$, it holds that
\begin{align*}
\label{}
\normt{\hat{\phi}}{L^r(\R)} \leq (2\pi)^{-\theta/2}\norm{\phi}{L^q(\R)}  , \quad \phi\in L^q(\R).
\end{align*}
\end{lem}

The following result is essentially stated in \citet[equation~(A.7)]{ChGrLaSp:11}, but we need to restate it here as we are working with a $k$-dependent norm and want $k$-explicit estimates.
\begin{lem}
\label{LpHsThm}
For $1\leq q\leq 2$ and $s<1/2-1/q$, $L^q(\R)$ can be continuously embedded in $H^s(\R)$, with
\begin{align}
\label{LpHsBound}
\normt{\phi}{H^s_k(\R)} \leq c(s,k,\theta)\norm{\phi}{L^q(\R)}  , \quad \phi\in L^q(\R),
\end{align}
where $\theta$ is as defined in Lemma~\ref{FTLpThm} and
\begin{align*}
\label{}
c(s,k,\theta) = \left(\frac{1}{2\pi}\int_{-\infty}^\infty (k^2+\xi^2)^{s/\theta}\,\rd \xi \right)^{\theta/2} .
\end{align*}
\end{lem}

\begin{proof}
By the density of $C_0^\infty(\R)$ in $L^q(\R)$ it suffices to prove \rf{LpHsBound} for $\phi\in C_0^\infty(\R)$. Let $\phi\in C_0^\infty(\R)$, let $1<q\leq 2$ (the case $q=1$ requires an obvious trivial modification of the proof), let $r$ be such that $1/q+1/r=1$, and let $\theta = 2/q-1$ as in Lemma~\ref{FTLpThm}. 
Provided that $s<1/2-1/q$, we have $s/\theta<-1/2$, so that the function $(k^2+\xi^2)$ is in $L^{1/\theta}(\R)$, and H\"older's inequality gives
\begin{align*}
\label{}
\normt{\phi}{H^s_k(\R)}^2 = \int_{\R} (k^2+\xi^2)^s |\hat{\phi}(\xi)|^2 \,\rd \xi
& \leq \left(\int_{\R} (k^2+\xi^2)^{s/\theta} \,\rd \xi\right)^{\theta}  \left(  \int_{\R} (|\hat{\phi}(\xi)|^2)^{r/2} \,\rd \xi\right)^{2/r} \notag \\
& = c(s,k,\theta)^2 (2\pi)^\theta  \normt{\hat{\phi}}{L^r(\R)}^2 \notag \\
& \leq c(s,k,\theta)^2 \normt{\phi}{L^q(\R)}^2,
\end{align*}
the final inequality following from an application of Lemma~\ref{FTLpThm}.
\end{proof}

\begin{cor}
\label{LpHsCor}
For $1<q\leq 2$, $L^q(\R)$ can be continuously embedded in $H^{-1/2}(\R)$ with
\begin{align}
\label{LpHsBoundMinusHalf}
\normt{\phi}{H^{-1/2}_k(\R)} \leq k^{1/q-1} \max \{ 1,1/\sqrt{(2\pi-1)(q-1)}\} \norm{\phi}{L^q(\R)}  , \quad \phi\in L^q(\R).
\end{align}
\end{cor}

\begin{proof}
The result follows from Lemma~\ref{LpHsThm} with $s=-1/2$, combined with an explicit estimate of $c(-1/2,k,\theta)$. To derive this estimate we first note that
\begin{align}
c(-1/2,k,\theta) = \left(\frac{1}{2\pi}\int_{-\infty}^\infty (k^2+\xi^2)^{-1/(2\theta)} \,\rd \xi\right)^{\theta/2}%
& = k^{-1/2+\theta/2} \left(\frac{1}{\pi}\int_{0}^\infty (1+t^2)^{-1/(2\theta)} \,\rd t\right)^{\theta/2} \notag \\
& \leq k^{-1/2+\theta/2} \left(\frac{1}{\pi}\left( \int_{0}^1 \,\rd t + \int_{1}^\infty t^{-1/\theta} \,\rd t \right)\right) ^{\theta/2} \notag \\
& = k^{-1/2+\theta/2} \left(\frac{1}{\pi(1-\theta)}\right)^{\theta/2} \notag \\
& = k^{1/q-1}\left(\frac{q}{2\pi(q-1)}\right)^{1/q-1/2}.
\label{CEstMinusHalf}
\end{align}
Now define $q_*:=\pi/(\pi-1/2)$. 
For $q_*\leq q \leq 2$ we have that $q/(2\pi(q-1))\leq 1$. For $1<q<q_*$ we can estimate $q/(2\pi(q-1))\leq 1/((2\pi-1)(q-1))$. Inserting these estimates into \rf{CEstMinusHalf}, and noting that $0\leq 1/q -1/2 < 1/2$, we obtain \rf{LpHsBoundMinusHalf}.
\end{proof}

\begin{thm}
\label{HalfPlaneThmMain}
Suppose that a function $g(z)$ is analytic in $\real{z}>0$ and satisfies the bound
\begin{align*}
  |g(z)|\leq \hat{C}|z|^{-1/2}, \quad \real{z}>0,
\end{align*}
for some $\hat{C}>0$. Given $l>0$, $\alpha \in [0,1]$, and integers $n\geq 1$ and $p\geq 0$, let the degree vector $\bs{p}$ be defined by
\rf{piDef1}, and suppose that $n\geq cp$ for some constant $c>0$. Then 
for any $0<\epsilon<1/2$ there exists a constant $C>0$, depending only on $\epsilon$ and $\sigma$ (with $C\to\infty$ as $\epsilon\to 0 $ or $\epsilon\to 1/2$), and a constant $\tau>0$, depending only on $\epsilon$, $\sigma$, $\alpha$ and $c$ (with $\tau\to0$ as $\epsilon\to0$), such that
\begin{align}
\label{}
  \label{bestapphalfplaneCor4MainEps}
 \inf_{v\in \mathcal{P}_{\bp,n}(0,l)}\norm{g-v}{\tilde{H}^{-1/2}_k(0,l)}\leq  C \hat{C} k^{-1/2} (kl)^\epsilon\re^{-p\tau}.
\end{align}
\end{thm}

\begin{proof}
Our aim is to use Corollary~\ref{LpHsCor} to derive a best approximation error estimate in the $\tilde{H}^{-1/2}_k$ norm in terms of estimates in $L^q$ norms, $1<q< 2$. For the sharpest results (in terms of $k$-dependence) one might want to take $q=2$ in Corollary~\ref{LpHsCor}. However, this is not possible because $g$ cannot be assumed to be square integrable at $s=0$; this is why we assume that $1<q<2$.

We begin by defining a candidate approximant $V\in \mathcal{P}_{\bp,n}(0,l)$, which we take to be zero on $(0,x_1)$, and on $(x_{i-1},x_i)$, $i=2,\ldots,n$, to be equal to the $L^\infty$ best approximation to $g|_{(x_{i-1},x_i)}$ in $\mathcal{P}_{(\bp)_i}(x_{i-1},x_i)$, where $\mathcal{P}_{p}(a,b)$ denotes the space of polynomials of degree less than or equal to $p$ on the interval $(a,b)$. Then by Corollary~\ref{LpHsCor} we have
\begin{align}
  \label{bestapphalfplaneMainPrelim}
  \inf_{v\in \mathcal{P}_{\bp,n}(0,l)}\norm{g-v}{\tilde{H}^{-1/2}_k(0,l)}\leq \norm{g-V}{\tilde{H}^{-1/2}_k(0,l)} \leq  k^{1/q-1} \max \{ 1,1/\sqrt{(2\pi-1)(q-1)}\} \norm{g-V}{L^q(0,l)},
\end{align}
and it simply remains to estimate $\norm{g-V}{L^q(0,l)}$. To do this, we first note that
\begin{align}
\label{phi1Est}
\norm{g-V}{L^q(0,x_1)}\leq \hat{C} \left(\int_0^{x_1} s^{-q/2}\right)^{1/q} = 
\hat{C}D_1l^{1/q-1/2}\re^{-n\vartheta},
\end{align}
where $\vartheta:=(1/q-1/2)|\log{\sigma}|>0$ and  
\begin{align*}
D_1:=
\frac{\sigma^{1/2-1/q}}{(1-q/2)^{1/q}}.
\end{align*}
For any $i=2,\ldots,n$, $g$ is analytic in an ellipse containing the interval $(x_{i-1},x_i)$, and, using standard polynomial approximation results for analytic functions (see e.g.\ \citet[Lemma A.2]{convexpreprint} or \citet[Theorem 2.1.1]{Stenger}), one can show that
\begin{align*}
\norm{g-V}{L^\infty(x_{i-1},x_i)} \leq \hat{C}c_\sigma x_{i-1}^{-1/2}\re^{-(\bp)_i\eta}, \qquad i=2,\ldots,n,
\end{align*}
where $c_\sigma: = 2\sqrt{2}/\rho_\sigma>0$ and $\eta:=\log{\rho_\sigma}>0$, with $\rho_\sigma:=\left(1+\sigma^{1/2}(2-\sigma)^{1/2}\right)/(1-\sigma)$. Hence
\begin{align}
\label{phi2EstSingle}
\norm{g-V}{L^q(x_{i-1},x_i)}& \leq %
\hat{C}c_\sigma (x_i-x_{i-1})^{1/q-1/2}\left(\frac{x_i-x_{i-1}}{x_{i-1}}\right)^{1/2}  \re^{-(\bp)_i\eta}\notag\\
&=\hat{C}c_\sigma ((1-\sigma)l)^{1/q-1/2}\left(\frac{1-\sigma}{\sigma}\right)^{1/2} \re^{-(n-i)\vartheta-(\bp)_i\eta}.
\end{align}
Now, since
\begin{align*}
\label{}
(\bp)_i \geq \left(1-\alpha + \frac{\alpha(i-1)}{n}\right)p, \qquad i=2,\ldots,n,
\end{align*}
the sum
\begin{align*}
S:=\sum_{i=2}^{n}\left(  \re^{-(n-i)\vartheta-(\bp)_i\eta}\right)^q = \re^{q\vartheta} \sum_{i=2}^{n} \re^{-q((n+1-i)\vartheta+(\bp)_i\eta)}
\end{align*}
satisfies the estimate
\begin{align*}
S\leq \re^{q(\vartheta-(1-\alpha)p\eta)} \sum_{i=2}^{n} \re^{-q((n+1-i)\vartheta+(i-1)\mu)},
\end{align*}
where $\mu:=\alpha p\eta/n$. We then write
\begin{align*}
\label{}
(n+1-i)\vartheta + (i-1)\mu= (1/2)(n+1-i)\vartheta +n\psi(i),
\end{align*}
where $\psi(i):=(1/n)\left((n+1-i)\frac{\vartheta}{2}+(i-1)\mu\right)$. Since $\psi(i)$ is affine, we have
\begin{align*}
\label{}
\psi(i) \geq \min{\{\psi(1),\psi(n+1)\}} = \min{\left\{\vartheta/2,\mu\right\}}=:\nu,
\end{align*}
which gives
\begin{align*}
\label{}
S &\leq \re^{q(\vartheta-(1-\alpha)p\eta-n\nu)} \sum_{i=2}^{n} \re^{-(q/2)(n+1-i)\vartheta}
\leq \frac{\re^{q\vartheta/2}}{1-\re^{-q\vartheta/2}}\re^{-q(1-\alpha)p\eta}\re^{-qn\nu}.
\end{align*}
Combining this estimate with \rf{phi1Est} and \rf{phi2EstSingle} gives
\begin{align}
\label{phi2Estq}
\norm{g-V}{L^q(0,l)}\leq \hat{C} l^{1/q-1/2}\left( D_1\re^{-n\vartheta}+ D_2 \re^{-(1-\alpha)p\eta}\re^{-n\nu}\right),
\end{align}
where 
\begin{align*}
\label{}
D_2 := \frac{2\sqrt{2}(1-\sigma)^{1/q}\sigma^{-(1/(2q)-1/4)}}{\sigma(\sqrt{\sigma}+\sqrt{2-\sigma})(1-\sigma^{(1/2-q/4)})^{1/q}}.
\end{align*}
Finally, using \rf{phi2Estq} in \rf{bestapphalfplaneMainPrelim}, assuming $n\geq cp$, and letting $\epsilon:=1/q-1/2\in(0,1/2)$, gives \rf{bestapphalfplaneCor4MainEps} with 
\begin{align*}
\label{}
C = \max \{ 1,1/\sqrt{(2\pi-1)(q-1)}\} (D_1+D_2),%
\qquad
\tau  = \min\{c\vartheta,(1-\alpha)\eta + c\vartheta/2,\eta\}.
\end{align*}
Note that $D_1\to\infty$ and $\vartheta\to0$ as $q\to2$, i.e.\ as $\epsilon\to0$. 
\end{proof}

We are now ready to prove Theorem~\ref{dudnThm}.

\begin{proof}[Proof of Theorem~\ref{dudnThm}]Recalling \rf{eqn:varphi}, the result follows from Theorem~\ref{vpmThm} and Theorem~\ref{HalfPlaneThmMain}, with e.g.\ $C_3=2n_iC_1C$, where $C$ and $\tau$ are the constants obtained from applying Theorem~\ref{HalfPlaneThmMain} to $v_j^\pm$.
\end{proof}

\begin{rem}
We remark that Theorem~\ref{HalfPlaneThmMain} could also be used to achieve best approximation results in $H^{-1/2}$ for the problem of sound-soft scattering by a convex polygon, as considered in~\cite{HeLaMe:11} (where best approximation results are derived only in $L^2$).  %
\end{rem}

\section{Galerkin method}
\label{sec:gal}
Having designed an approximation space $V_{N,k}$ which can efficiently approximate $\varphi$, we now select an element of $V_{N,k}$ using the Galerkin method. That is, we seek $\varphi_N\in V_{N,k}\subset \tilde{H}^{-1/2}\left(\Gamma\right)$ such that (recall~\rf{BIE_sl} and~(\ref{eqn:varphi}))
\begin{equation}
  \left\langle S_k\varphi_N,v\right\rangle_{\Gamma}=\frac{1}{k}\left\langle f-S_k\Psi,v\right\rangle_{\Gamma},
 \quad\textrm{for all } v\in V_{N,k}.
  \label{eqn:gal}
\end{equation}
We note that since $\varphi_N,v\in V_{N,k}\subset L^2(\Gamma)$ the duality pairings in \rf{eqn:gal} can be evaluated simply as inner products in $L^2(\Gamma)$ (see the discussion after~(\ref{DualDef}) and the implementation details in \S\ref{sec:num}). %
Existence and uniqueness of the Galerkin solution $\varphi_N$ is guaranteed by the Lax-Milgram Lemma and Lemmas \ref{ThmSNormSmooth} and \ref{ThmCoerc}. Furthermore, by C\'{e}a's lemma we have the quasi-optimality estimate
\begin{equation}
  \norm{\varphi-\varphi_N}{\tilde{H}^{-1/2}_k(\Gamma)}\leq\frac{C_0 (1+\sqrt{kL})}{2\sqrt{2}}\inf_{v\in V_{N,k}}\norm{\varphi-v}{\tilde{H}^{-1/2}_k(\Gamma)}, %
\label{eqn:quasi-opt}
\end{equation}
where $C_0$ is the constant from Lemma~\ref{ThmSNormSmooth}. Combined with Theorem \ref{dudnThm}, this gives: 
\begin{thm}
\label{GalerkinThm}
Under the assumptions of Theorem~\ref{dudnThm}, we have
\begin{align}
\label{GalerkinErrorEst}
\norm{\varphi-\varphi_N}{\tilde{H}^{-1/2}_k(\Gamma)}\leq C_4\uM  k^{-1}(1+(kL)^{1/2+\epsilon})\,\re^{-p\tau}, %
\end{align}
where $C_4 = C_0C_3 /\sqrt{2}$ and $C_3$ is the constant from Theorem \ref{dudnThm}. %
Combined with Lemma \ref{lem:Mu} this implies the following $k$-explicit estimate, in which $C_5=3C_4C_2$: %
\begin{align}
\label{GalerkinErrorEst_kexplicit}
\norm{\varphi-\varphi_N}{\tilde{H}^{-1/2}_k(\Gamma)}\leq C_5k^{-1}(1+(kL)^{3/2+\epsilon})\,\re^{-p\tau}. %
\end{align}
\end{thm}

An approximation $u_N$ to the solution $u$ of the BVP $\sP$ can be found by inserting the approximation $[\partial u/\partial\bn] \approx \Psi+ k \varphi_N$ into the formula~\rf{eqn:D_RepThm}, i.e.
\begin{align*}
  u_N(\bx) := u^i(\bx) - \int_{\Gamma}\Phi_k(\bx,\by)\left(\Psi(\by) +  k\varphi_N(\by)\right)\,\rd s(\by), \qquad \bx\in D.
\end{align*}
We then have the following (nonuniform) error estimate:
\begin{thm}
\label{thm:domainerror}
Under the assumptions of Theorem~\ref{dudnThm}, we have, for $\bx\in D$ and $d=\dist{(\bx,\Gamma)}$,
\begin{align}
\label{GalerkinErrorDomain}
\frac{|u(\bx)-u_N(\bx)|}{\norm{u}{L^\infty(D)}} \leq C_6 \left(1\!+\!\frac{1}{\sqrt{kd}}\right)\log\left( 2+\frac{1}{kd}\right) \log^{1/2}(2+kL)(1\!+\!(kL)^{1/2+\epsilon})\,\re^{-p\tau}, %
\end{align}
where $C_6=(1+1/\sqrt{k\ell_{min}})C_4C$, and $C$ is the constant from Lemma \ref{lem:H_one_half_norms}(ii).
\end{thm}
\begin{proof}
Noting that $|u(\bx)-u_N(\bx)| = k |\cS_k(\varphi-\varphi_N)(\bx)| \leq k \norm{\Phi_k(\bx,\cdot)}{H^{1/2}_k(\Gamma)}\norm{\varphi-\varphi_N}{\tilde{H}^{-1/2}_k(\Gamma)}$, the result follows from Lemma~\ref{lem:H_one_half_norms}(ii) and \rf{GalerkinErrorEst}.
\end{proof}

An object of interest in applications is the \emph{far field pattern} of the scattered field. An asymptotic expansion of the representation~\rf{eqn:D_RepThm} reveals that (cf.~\cite{CoKr:92})%
\begin{equation*}
  u^s(\bx) \sim \dfrac{\re^{\ri\pi/4}}{2\sqrt{2\pi}}\dfrac{\re^{\ri kr}}{\sqrt{kr}}F(\hat{\bx}), \quad \mbox{as }r:=|\bx|\rightarrow\infty,
\end{equation*}
where $\hat{\bx}:=\bx/|\bx|\in \mathbb{S}^1$, the unit circle, and
\begin{align}
  \label{FDef}
  F(\hat{\bx}):=-\int_\Gamma \re^{-\ri k\hat{\bx}\cdot \by}\left[\pdone{u}{\bn}\right]\!(\by)\,\rd s(\by), \qquad \hat{\bx}\in \mathbb{S}^1.
\end{align}
An approximation $F_N$ to the far field pattern $F$ can be found by inserting the approximation $[\partial u/\partial\bn] \approx \Psi + k\varphi_N$ into the formula~\rf{FDef}, i.e.
\begin{equation}
  F_N(\hat{\bx}):=-\int_\Gamma \re^{-\ri k\hat{\bx}\cdot \by} \left(\Psi(\by)+ k\varphi_N(\by)  \right) \,\rd s(\by), \qquad \hat{\bx}\in \mathbb{S}^1.
  \label{eqn:FFP_approx}
\end{equation}

\begin{thm}
\label{FarFieldThm} Under the assumptions of Theorem~\ref{dudnThm} we have
\begin{align}
  \label{FarFieldErrorEst}
  \norm{F-F_N}{L^\infty(\mathbb{S}^1)}\leq C_7(1+(kL)^{2+\epsilon})\,\re^{-p\tau}, %
\end{align}
where $C_7=3C_5C$ and $C$ is the constant from Lemma~\ref{lem:H_one_half_norms}(i).
\end{thm}
\begin{proof}
Noting that $|F(\hat{\bx})-F_N(\hat{\bx})| = k|\langle \re^{-\ri k\hat{\bx}\cdot(\cdot)},\varphi-\varphi_N\rangle_\Gamma| \leq k \normt{\re^{-\ri k\hat{\bx}\cdot(\cdot)}}{H^{1/2}_k(\Gamma)}\norm{\varphi-\varphi_N}{\tilde{H}^{-1/2}_k(\Gamma)}$, the result follows from Lemma~\ref{lem:H_one_half_norms}(i) and~\rf{GalerkinErrorEst_kexplicit}.
\end{proof}

Similarly, approximations $u'_N$ and $F_N'$ to the solution $u'$ of the BVP $\sP'$ and the far-field pattern $F'$ associated with the diffracted field $u^d$ can be found using \rf{eqn:Dp_RepThm}, and estimates similar 
to \rf{GalerkinErrorDomain} and \rf{FarFieldErrorEst} can be proved. 

\begin{rem} The algebraically $k$-dependent factors in the error estimates~\rf{GalerkinErrorEst}, \rf{GalerkinErrorEst_kexplicit}, \rf{GalerkinErrorDomain} and~\rf{FarFieldErrorEst} can be absorbed into the exponentially decaying factors by allowing $p$ to grow in proportion to $\log{k}$ as $k\to\infty$ (cf.\ \cite[Rem.\ 6.5]{HeLaMe:11}). Therefore, since $N\sim p^2$, our estimates show that to maintain any desired accuracy it is sufficient to increase $N$ in proportion to $\log^2{k}$ as $k\to\infty$, as claimed in \S\ref{Intro}. In fact, our numerical results in \S\ref{sec:num} below suggest that in practice this logarithmic increase is not required, and that when computing $u$ or $F$ with a fixed number of degrees of freedom, accuracy actually improves as frequency increases.
\end{rem}
\section{Numerical results}
\label{sec:num}
We present numerical computations of the Galerkin approximation $\varphi_N$, as defined by~(\ref{eqn:gal}), for the screen/aperture shown in Figures~\ref{fig:screen_domain} and ~\ref{fig:aperture_domain}. 
Our results confirm our theoretical predictions, demonstrating the efficacy and efficiency of our method, and its robustness across a wide range of frequencies. %

The screen $\Gamma$ we consider has multiple components of different lengths and different separations, and is defined explicitly by \rf{eqn:GammaDef} with $n_i=5$, $s_1=0$, $s_2=2\pi$, $s_3=21\pi/10$, $s_4=5\pi/2$, $s_5=14\pi/5$, $s_6=7\pi/2$, $s_7=4\pi$, $s_8=6\pi$, $s_9=61\pi/10$ and $s_{10}=10\pi$. Hence $L_1=2\pi$, $L_2=2\pi/5$, $L_3=7\pi/10$, $L_4=2\pi$ and $L_5=39\pi/10$, so that the smallest component has length $2\pi/5$, the longest has length $39\pi/10$, and the sum of the length of all of the components is $\sum_{i=1}^{n_i}L_j = 9\pi = 9k\lambda/2$ (where $\lambda=2\pi/k$ is the wavelength). We present results below for values of $k$ ranging from $k=10$ (in which case the smallest segment is two wavelengths long) up to $k=10240$ (in which case the longest segment is nearly 20000 wavelengths long).  
The plots in Figures~\ref{fig:screen_domain} and~\ref{fig:aperture_domain} show the total fields for the ``non-grazing'' incident direction $\bd=(1/\sqrt{2}, -1/\sqrt{2})$; in our examples below we also consider the ``grazing'' incident direction $\bd=(1, 0)$.

To describe our implementation of the HNA approximation space of \S\ref{ApproxSpace}, 
we write $\varphi_N\in V_{N,k}$ as 
\begin{equation}
  \varphi_N(\cdot) = \sum_{\ell=1}^N v_{\ell} \chi_{\ell}(\cdot),
  \label{eqn:varphi_N}
\end{equation}
where $N$ is given by~(\ref{eqn:dof1}), $v_{\ell}$, $\ell=1,\ldots,N$, are the unknown coefficients to be determined, and $\chi_{\ell}$, 
$\ell=1,\ldots,N$, are the HNA basis functions, which we now define. Each basis function $\chi_{\ell}$ is supported on an interval $(a,b)\subset(s_{2j-1},s_{2j})$ for some $j\in\{1,\ldots,n_i\}$, and takes the form
\[\chi_{\ell}(s)= \sqrt{\frac{2q+1}{b-a}} P_q\left(2\left(\frac{s-a}{b-a}\right)-1\right)\re^{\rho \ri k s}, \quad s\in (a,b), \]
where $P_q$, $q\leq p$, denotes the Legendre polynomial of order $q$, and either $\rho=1$ and $a=s_{2j-1}+x_m$, $b=s_{2j-1}+x_{m+1}$, for some $j\in\{1,\ldots,n_i\}$ and $m\in\{0,\ldots,n_j^+\}$ (in which case $\chi_{\ell}$ is one of the basis functions used to approximate the amplitude $v_j^+$ in \rf{eqn:varphi} (see Figure~\ref{fig:grading})), or else $\rho=-1$ and $a=s_{2j}-x_{m+1}$, $b=s_{2j}-x_{m}$, for some $j\in\{1,\ldots,n_i\}$ and $m\in\{0,\ldots,n_j^-\}$ (in which case $\chi_{\ell}$ is one of the basis functions used to approximate the amplitude $v_j^-$ in \rf{eqn:varphi} (again, see Figure~\ref{fig:grading})), with $x_m$, $m=0,\ldots,n_j^{\pm}$ defined as in \rf{mesh_graded}. 
This choice means that $\left\langle \chi_j,\chi_j\right\rangle_{\Gamma}=1$, $j=1,\ldots,N$, and that $\left\langle \chi_j,\chi_m\right\rangle_{\Gamma}=0$, $j\neq m$, unless $\chi_j$ and $\chi_m$ are supported on non-identical overlapping intervals.  Substituting~(\ref{eqn:varphi_N}) into~(\ref{eqn:gal}) produces the linear system
\begin{equation}
  \sum_{\ell=1}^{N} \left\langle S_k\chi_{\ell},\chi_m\right\rangle_{\Gamma} v_{\ell} =\frac{1}{k}\left\langle f-S_k\Psi,\chi_m\right\rangle_{\Gamma},
 \quad m=1,\ldots,N.
  \label{eqn:linsys}
\end{equation}
To construct \rf{eqn:linsys} it is necessary to evaluate oscillatory integrals.  Due to the linear nature of the screen, these are computed efficiently using Filon quadrature, as described in \citet[\S4]{ChGrLaSp:11} and \citet[\S4]{Tw:13}. 
In all our experiments we take $\alpha=1$ and $n_j^+=n_j^-=2(p+1)$, $j=1,\ldots,n_i$.  Experiments in \citet[\S3]{Tw:13} for the case $n_i=1$ suggest that, for the examples considered therein, these choices are appropriate in terms of attempting to minimise the number of degrees of freedom required to achieve a prescribed level of accuracy.  Using~\rf{eqn:dof1}, the total number of degrees of freedom is then: $N=20$ for $p=0$, $N=70$ for $p=1$, $N=130$ for $p=2$, $N=220$ for $p=3$, $N=320$ for $p=4$, $N=450$ for $p=5$, $N=590$ for $p=6$, and $N=760$ for $p=7$ (these values are the same for all values of $k$). In each case, the linear system \rf{eqn:linsys} is inverted using a standard direct solver.

Theorem~\ref{GalerkinThm} predicts exponential decay of $\norm{\varphi-\varphi_N}{\tilde{H}_k^{-1/2}(\Gamma)}$ as $p$ increases, for fixed $k$, and moreover that increasing $p$ proportionally to $\log k$ as $k$ increases should be sufficient to maintain accuracy.  In practice, it is not straightforward to compute $\norm{\cdot}{\tilde{H}_k^{-1/2}(\Gamma)}$; instead, we compute $\|\cdot\|_{\Gamma}$, defined by
\begin{align*}
\label{}
\|\phi\|_{\Gamma}:= \sqrt{|\langle S_k\phi,\phi \rangle_{\Gamma}|}, \qquad \phi\in\tilde H^{-1/2}(\Gamma),
\end{align*}
which defines an equivalent norm on $\tilde H^{-1/2}(\Gamma)$ and is easier to compute (see, e.g., the discussion in \citet[pp.~A:29--A:30]{SmArBePhSc:13}).  Specifically, it follows from \rf{SEst2D} and~\rf{SCoercive} that
\[ \frac{1}{\sqrt{2\sqrt{2}}} \|\phi\|_{\tilde{H}_k^{-1/2}(\Gamma)} \leq 
\|\phi\|_\Gamma
\leq \sqrt{C_0(1+\sqrt{kL})} \|\phi\|_{\tilde{H}_k^{-1/2}(\Gamma)}, 
\qquad \phi\in\tilde{H}^{-1/2}(\Gamma),
\]
and hence combining the right inequality with Theorem~\ref{GalerkinThm} we expect
\begin{equation}
\|\varphi-\varphi_N\|_\Gamma
 \leq C_5\sqrt{C_0(1+\sqrt{kL})} k^{-1}(1+(kL)^{3/2+\epsilon})\,\re^{-p\tau}.
  \label{eqn:equiv_norm_est}
\end{equation}
In fact, we will see below that as $k$ increases with $p$ fixed, $\|\varphi-\varphi_N\|_\Gamma$ actually decreases, suggesting that we can maintain accuracy as $k\to\infty$ with a fixed number of degrees of freedom.  Similarly, we will see that the relative error, $\|\varphi-\varphi_N\|_\Gamma/\|\varphi\|_\Gamma$, grows only very slowly as $k$ increases with $N$ fixed.  We will also compute the solution in the domain and the far field pattern, making comparison with the error estimates (\ref{GalerkinErrorDomain}) and~(\ref{FarFieldErrorEst}).

Since $N$ depends only on $p$, and the values of $p$ are more intuitively meaningful, we introduce the additional notation $\psi_p(s):=\varphi_{N}(s)$.  We begin in Figure~\ref{fig:bdy_soln} by plotting $|\psi_7(s)|$ (sampled at 500,000 evenly spaced points on the boundary) for both grazing and non-grazing incidence, for $k=10$ and $k=2560$.%
\begin{figure}[htb]
\begin{center}
\subfigure[non-grazing, $k=10$]{\includegraphics[width=5.9cm]{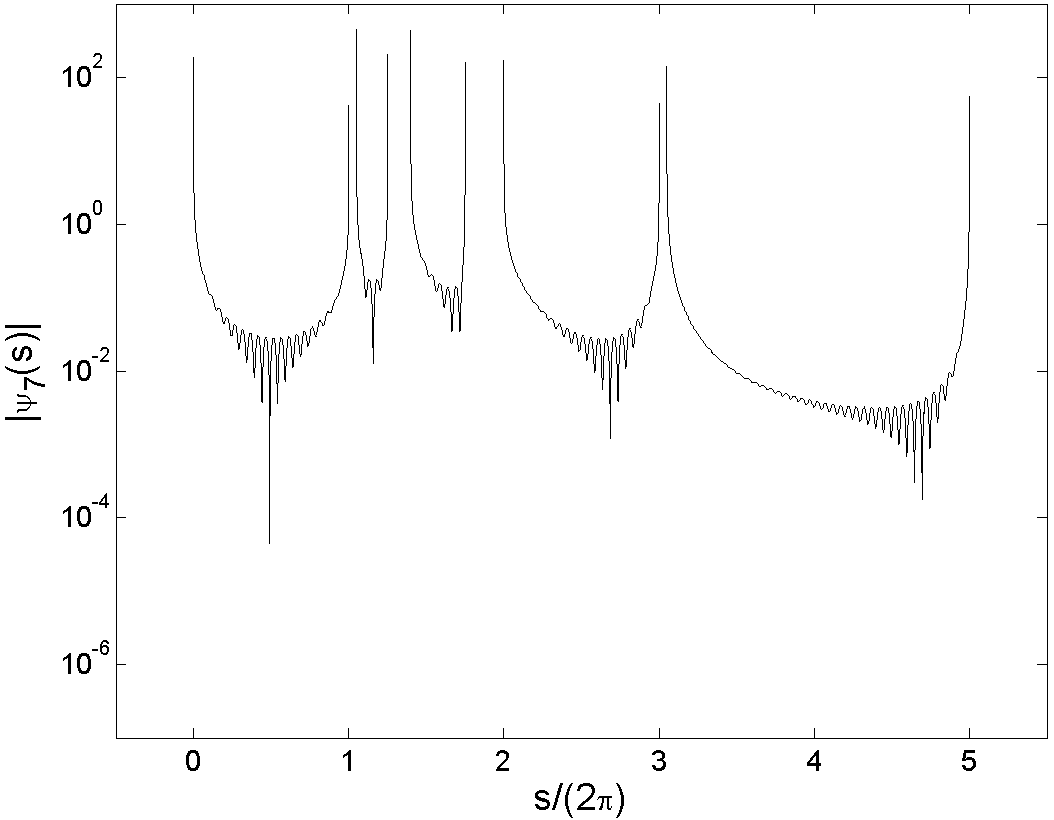}}
\hs{2}
\subfigure[grazing, $k=10$]{\includegraphics[width=5.9cm]{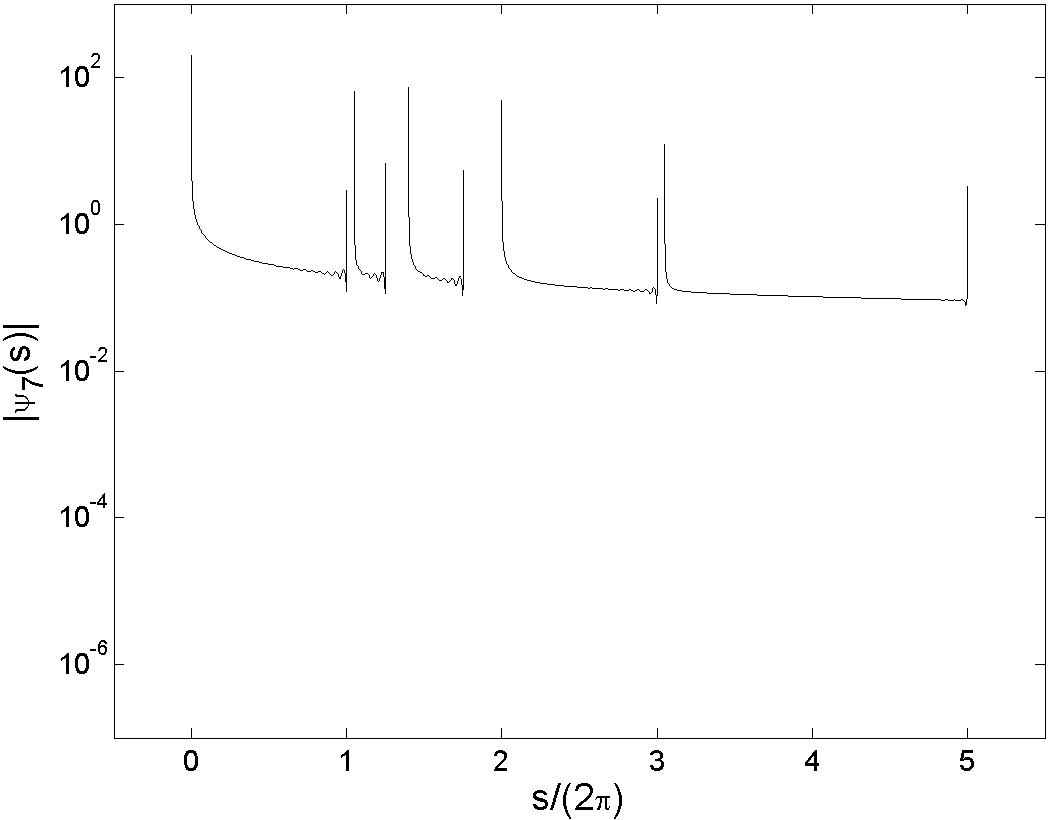}}
\vs{1}
\subfigure[non-grazing, $k=2560$]{\includegraphics[width=5.9cm]{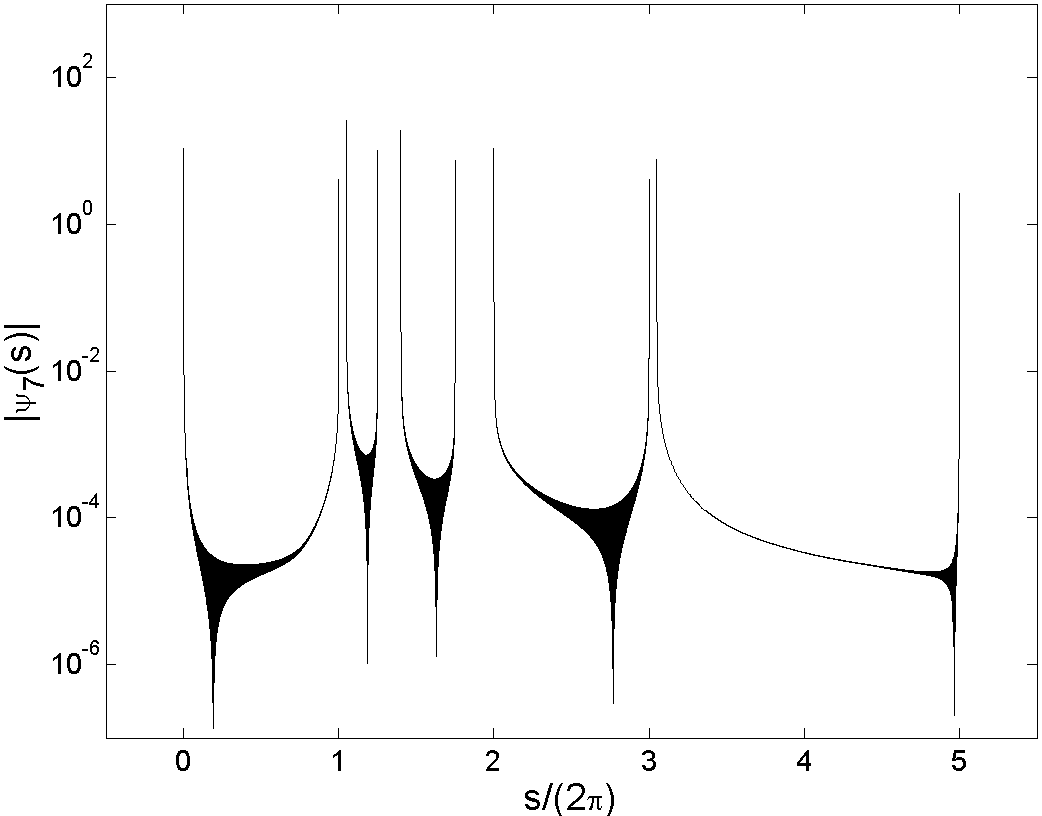}}
\hs{2}
\subfigure[grazing, $k=2560$]{\includegraphics[width=5.9cm]{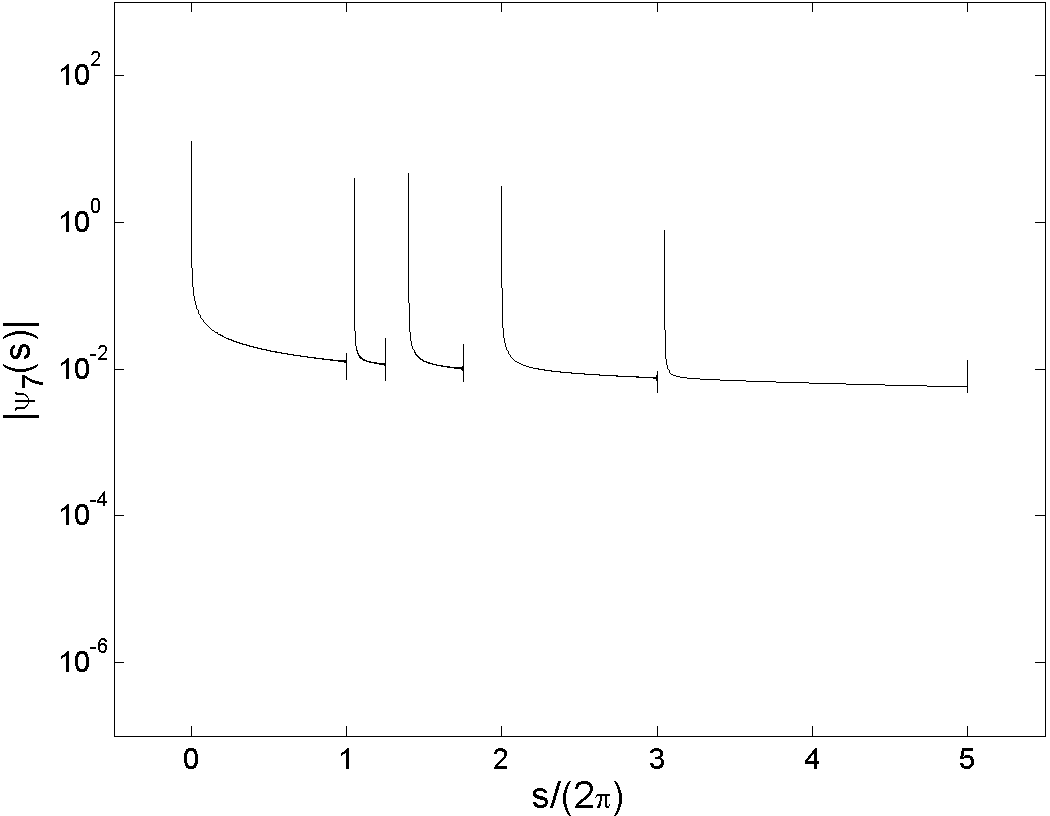}}
\end{center}
\caption{Boundary solution for grazing and non-grazing incidence, with $k=10$ and $k=2560$.}
\label{fig:bdy_soln}
\end{figure}
There is a singularity in the solution $\varphi$ at the edge of each component of the screen.  These singularities are evident in Figure~\ref{fig:bdy_soln} as is the increased oscillation for larger~$k$.  (The apparent shaded region is an artefact of very high oscillation.) 

In Figure~\ref{fig:bdy_errors} we plot the error 
$e_p:=\|\psi_7-\psi_p\|_\Gamma$ 
and the relative error 
$ r_p:=\|\psi_7-\psi_p\|_\Gamma /\|\psi_7\|_\Gamma$   
against $p$, for grazing and non-grazing incidence, for a range of values of $k$.
\begin{figure}[htb]
\begin{center}
\subfigure[$e_p$, non-grazing incidence]{\includegraphics[width=5.9cm]{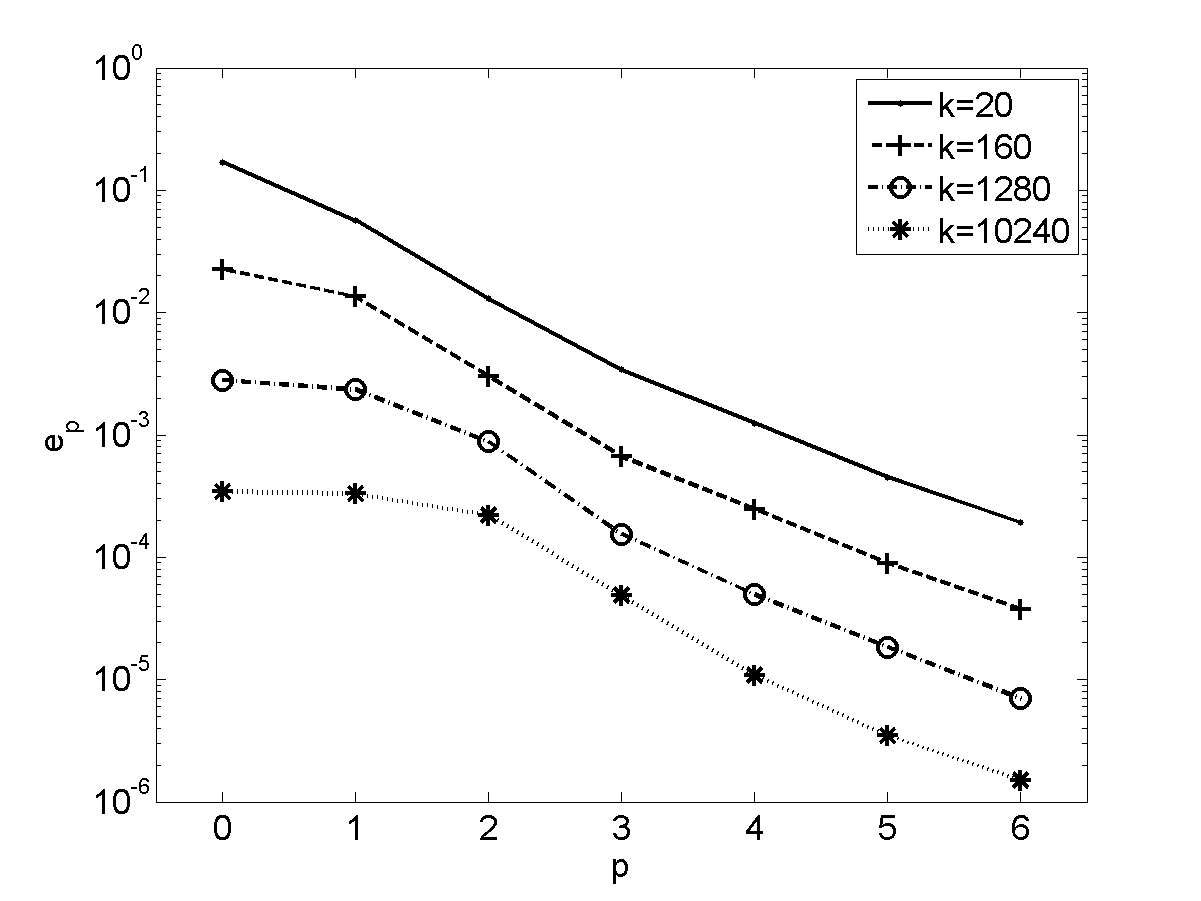}}
\hs{1}
\subfigure[$e_p$, grazing incidence]{\includegraphics[width=5.9cm]{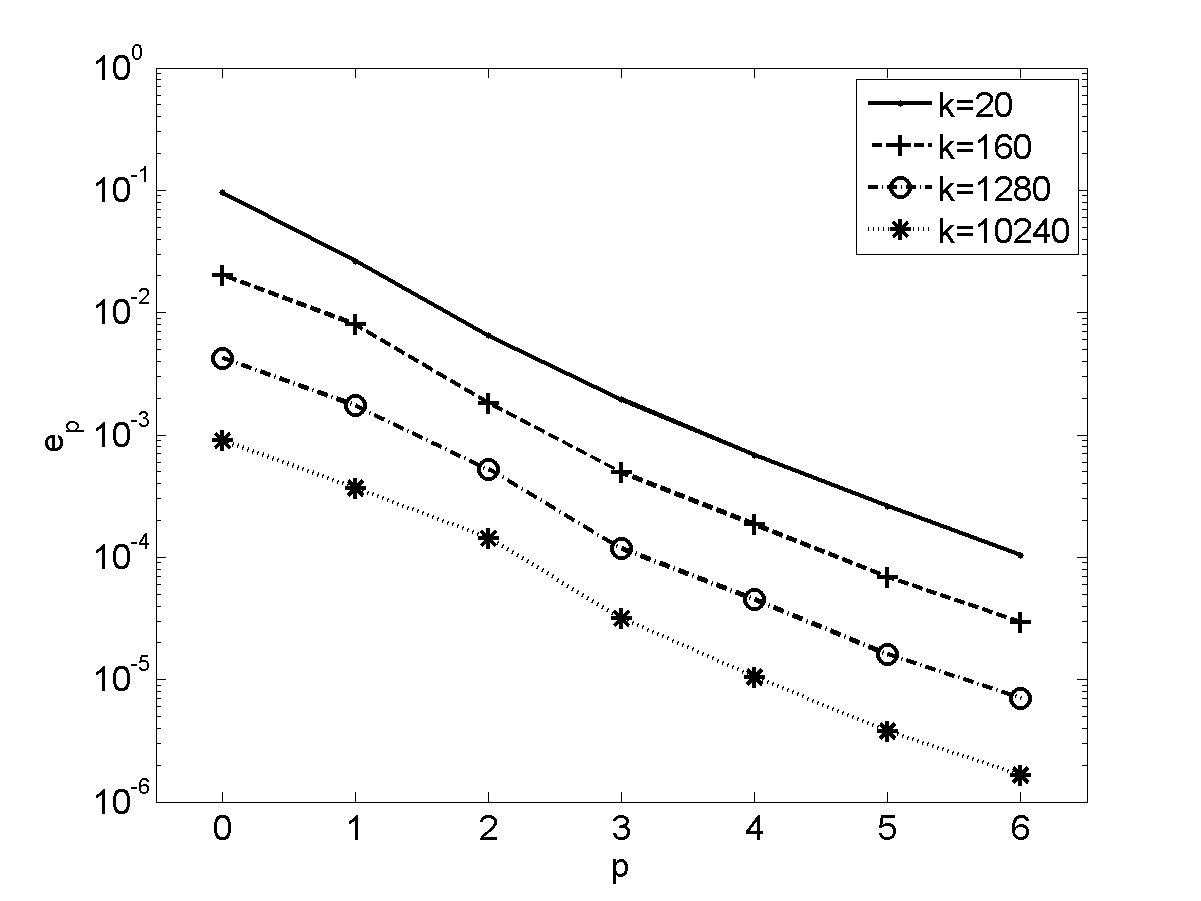}}
\vs{1}
\subfigure[$r_p$, non-grazing incidence]{\includegraphics[width=5.9cm]{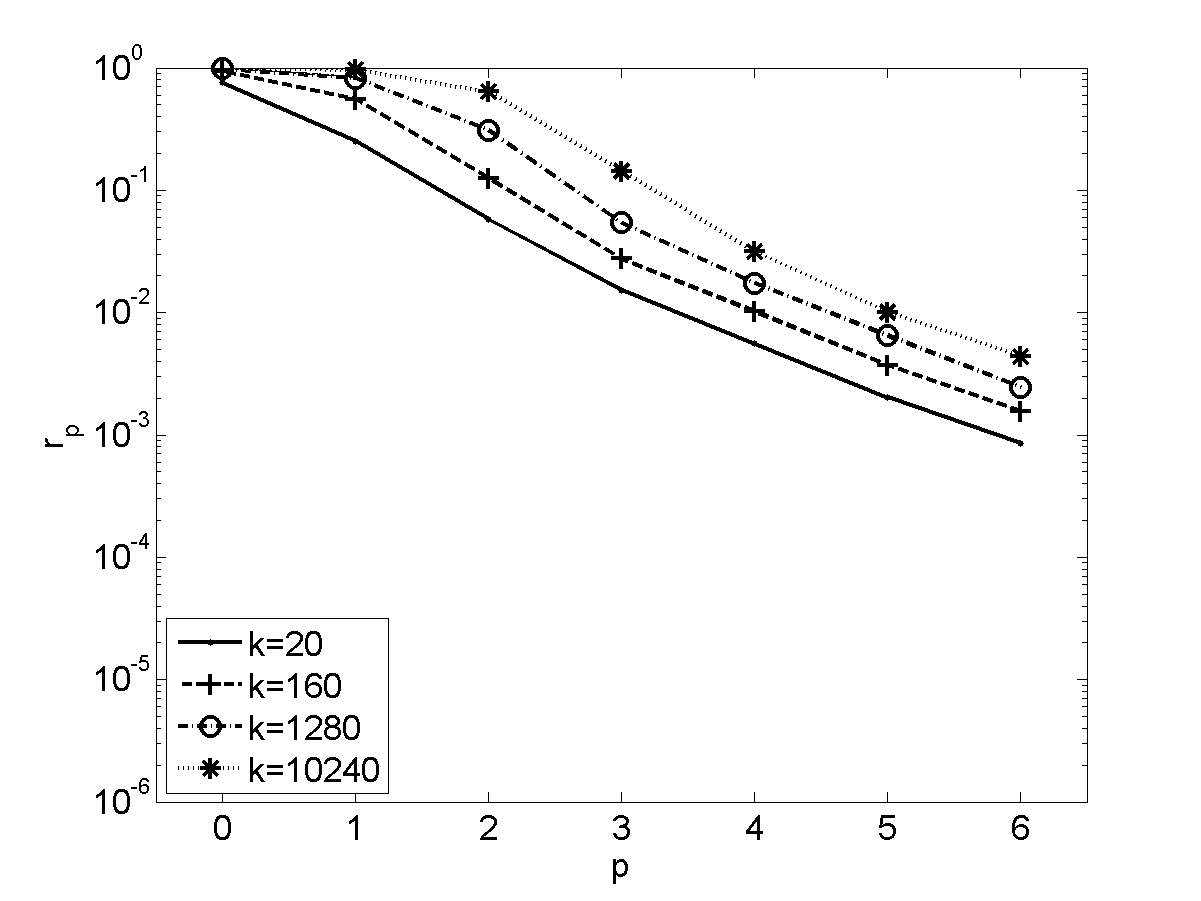}}
\hs{1}
\subfigure[$r_p$, grazing incidence]{\includegraphics[width=5.9cm]{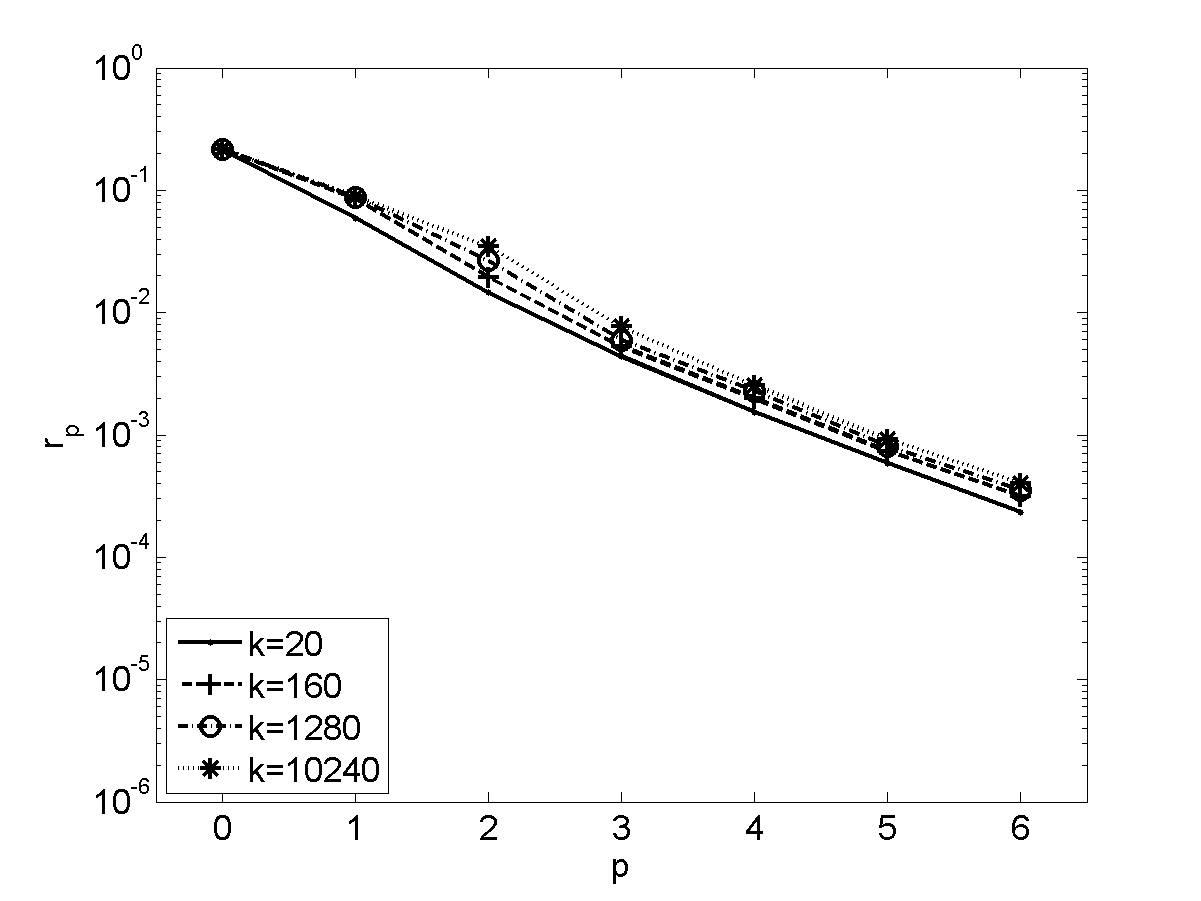}}
\end{center}
\caption{Errors and relative errors in the boundary solution.}
\label{fig:bdy_errors}
\end{figure}
We take the ``exact'' reference solutions to be those computed with $p=7$, as plotted for $k=10$ and $k=2560$ in Figure~\ref{fig:bdy_soln}.%

Figure~\ref{fig:bdy_errors} shows exponential decay as $p$ increases for both incident angles and for all values of $k$, as predicted by~(\ref{eqn:equiv_norm_est}).  A further key question is how the accuracy depends on $k$.  For both incident angles, the plots in the upper half of Figure~\ref{fig:bdy_errors} suggest that the errors decrease as $k$ increases, whilst the plots in the lower half of Figure~\ref{fig:bdy_errors} suggest that the relative errors increase only very slowly as $k$ increases. To investigate this further, in Table~\ref{table1} we show results for the two angles of incidence for $p=5$ (and hence $N=450$), for a wider range of values of $k$. We tabulate errors $e_p$, relative errors $r_p$, and also $N/(\sum_{j=1}^{n_i}L_j/\lambda)=2N/9k$, the average number of degrees of freedom per wavelength.  As $k$ increases the absolute error $e_p$ decreases, as shown in Figure~\ref{fig:bdy_errors}, and the relative error $r_p$ increases only very slowly, while the average number of degrees of freedom per wavelength decreases in proportion to~$k^{-1}$.%
\begin{table}[t!]
  \begin{center}
  { \setlength{\tabcolsep}{2.5mm}
    \begin{tabular}{|c|r|r|c|c|c|c|c|}
       \hline 
      $\bd$ & $k$ & $\frac{N}{\sum_{j=1}^{n_i}L_j/\lambda}$  & $e_p$ & $\mu$ & $r_p$ & COND & rel cpt \\
      \hline
 $\left(\frac{1}{\sqrt{2}},\frac{-1}{\sqrt{2}}\right)$  
                             &   10 &  10.00 &   9.25$\times10^{-4}$ & -1.03 &   2.18$\times10^{-3}$ &   1.50$\times10^{9}$ & 1.00 \\ %
                             &   20 &   5.00 &   4.51$\times10^{-4}$ & -0.77 &   2.01$\times10^{-3}$ &   1.03$\times10^{9}$ & 0.98 \\ %
                             &   40 &   2.50 &   2.64$\times10^{-4}$ & -0.87 &   2.49$\times10^{-3}$ &   7.12$\times10^{8}$ & 0.98 \\ %
                             &   80 &   1.25 &   1.45$\times10^{-4}$ & -0.69 &   2.88$\times10^{-3}$ &   4.97$\times10^{8}$ & 0.99 \\ %
                             &  160 &   0.63 &   8.99$\times10^{-5}$ & -0.80 &   3.72$\times10^{-3}$ &   3.50$\times10^{8}$ & 1.00 \\
                             &  320 &   0.31 &   5.16$\times10^{-5}$ & -0.74 &   4.32$\times10^{-3}$ &   2.47$\times10^{8}$ & 0.99 \\ %
                             &  640 &   0.16 &   3.08$\times10^{-5}$ & -0.74 &   5.08$\times10^{-3}$ &   1.75$\times10^{8}$ & 1.00 \\
                             & 1280 &   0.08 &   1.85$\times10^{-5}$ & -0.67 &   6.48$\times10^{-3}$ &   1.23$\times10^{8}$ & 1.00 \\
                             & 2560 &   0.04 &   1.16$\times10^{-5}$ & -0.91 &   7.64$\times10^{-3}$ &   8.72$\times10^{7}$ & 1.00 \\
                             & 5120 &   0.02 &   6.18$\times10^{-6}$ & -0.83 &   9.14$\times10^{-3}$ &   6.17$\times10^{7}$ & 1.01 \\
                            & 10240 &   0.01 &   3.47$\times10^{-6}$ &       &   1.01$\times10^{-2}$ &   4.36$\times10^{7}$ & 1.01 \\
   \hline
 $(1,0)$  %
            &   10 &  10.00 &   3.39$\times10^{-4}$ & -0.38 &   4.52$\times10^{-4}$ &   1.50$\times10^{9}$ & 1.00 \\
            &   20 &   5.00 &   2.60$\times10^{-4}$ & -0.61 &   5.84$\times10^{-4}$ &   1.03$\times10^{9}$ & 1.01 \\
            &   40 &   2.50 &   1.70$\times10^{-4}$ & -0.60 &   6.43$\times10^{-4}$ &   7.12$\times10^{8}$ & 0.99 \\
            &   80 &   1.25 &   1.12$\times10^{-4}$ & -0.71 &   7.13$\times10^{-4}$ &   4.97$\times10^{8}$ & 0.98 \\
            &  160 &   0.63 &   6.84$\times10^{-5}$ & -0.69 &   7.31$\times10^{-4}$ &   3.50$\times10^{8}$ & 0.99 \\
            &  320 &   0.31 &   4.23$\times10^{-5}$ & -0.68 &   7.59$\times10^{-4}$ &   2.47$\times10^{8}$ & 0.99 \\
            &  640 &   0.16 &   2.64$\times10^{-5}$ & -0.72 &   7.97$\times10^{-4}$ &   1.75$\times10^{8}$ & 1.00 \\
            & 1280 &   0.08 &   1.60$\times10^{-5}$ & -0.62 &   8.13$\times10^{-4}$ &   1.23$\times10^{8}$ & 1.00 \\
            & 2560 &   0.04 &   1.04$\times10^{-5}$ & -0.73 &   8.92$\times10^{-4}$ &   8.72$\times10^{7}$ & 1.00 \\
            & 5120 &   0.02 &   6.27$\times10^{-6}$ & -0.73 &   9.02$\times10^{-4}$ &   6.17$\times10^{7}$ & 1.01 \\
           & 10240 &   0.01 &   3.78$\times10^{-6}$ &       &   9.14$\times10^{-4}$ &   4.36$\times10^{7}$ & 1.00 \\
   \hline
    \end{tabular}
    }
  \end{center}
  \caption{Errors $e_p$ and relative errors $r_p$, for non-grazing ($\bd=(1/\sqrt{2},-1/\sqrt{2})$) and grazing ($\bd=(1,0)$) incidence, with $p=5$ (and hence $N=450$).}
  \label{table1}
\end{table}
We also tabulate $\log_2(e_p(2k)/e_p(k))$, where $e_p(k)$ refers to the absolute error $e_p$ for a particular value of $k$. This is an estimate of the order of convergence, $\mu$, on a hypothesis that $e_p(k) \sim k^{\mu}$ as $k\rightarrow\infty$.  The values of $\mu\in(-0.91,-0.60)$ for $k\geq 20$ are considerably lower than might be anticipated from the estimate~(\ref{eqn:equiv_norm_est}), suggesting that our estimates are not sharp in terms of their $k$-dependence.  In particular, the results are consistent with the conjecture that $\uM=\ord{1}$ (as discussed just before Lemma~\ref{lem:Mu}).

In Table~\ref{table1} we also show the condition number (COND) of the $N$-dimensional linear system \rf{eqn:linsys}, and 
we investigate the dependence of the condition number on both $k$ and $p$ further in Figure~\ref{fig:cond}.
\begin{figure}[p]
\begin{center}
  \includegraphics[height=6cm]{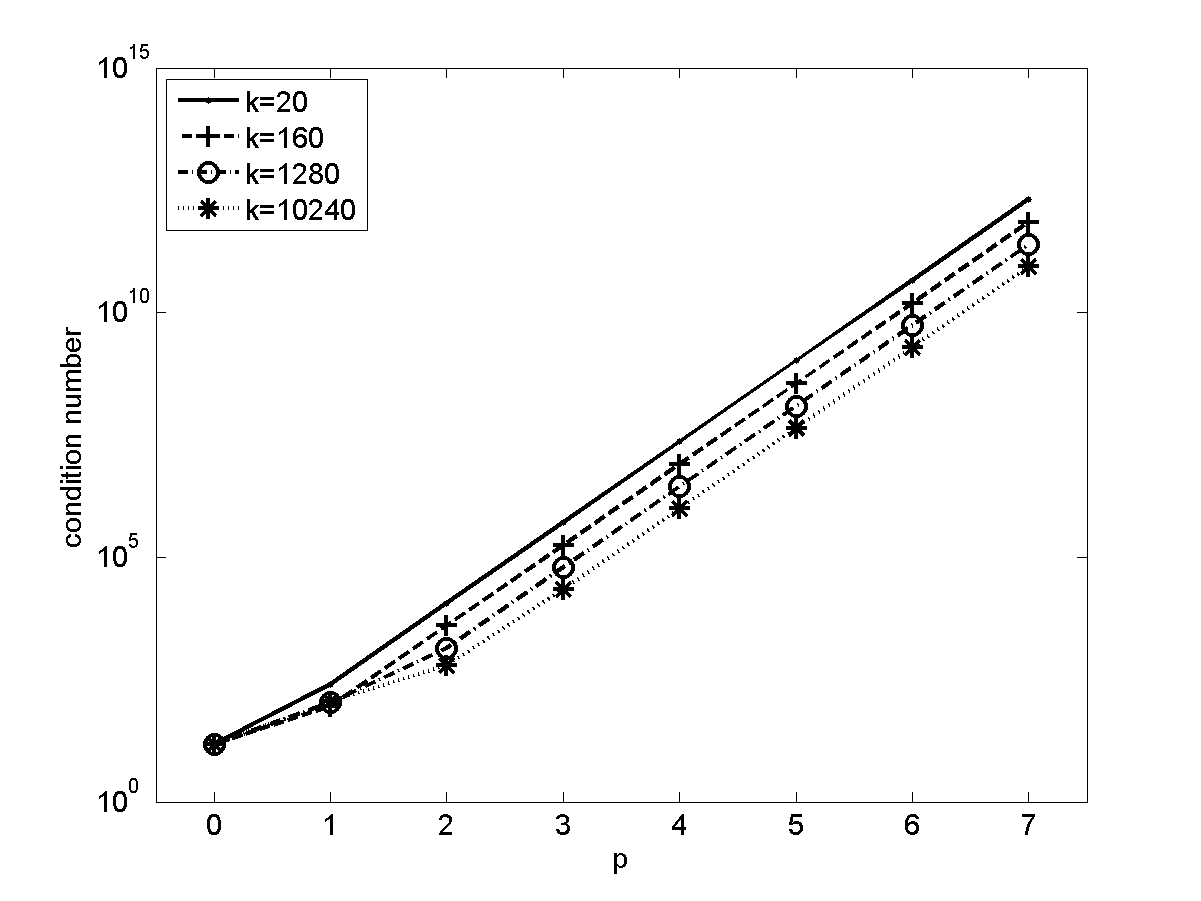}
\end{center}
\caption{Condition number of the $N$-dimensional linear system \rf{eqn:linsys}.}
\label{fig:cond}
\end{figure}
For fixed $k$, the condition number grows exponentially with respect to $p$ (note the logarithmic scale on the vertical axis).  This rapid growth in the condition number as $p$ increases is not surprising: for weakly singular BIEs of the first kind, the condition number for standard $hp$ Galerkin BEM, with a geometrically graded mesh (as used here), is known to grow exponentially with respect to the number of unknowns (see, e.g., \citet{HeStTr:98}).  For fixed $p$, the condition number decreases slowly as $k$ increases (and hence as the average number of degrees of freedom per wavelength decreases), and we note that the condition numbers we encountered in our experiments were not so large as to cause problems for our direct solver.
Furthermore, as remarked in~\S\ref{Intro}, our best approximation results (though not our full analysis) hold regardless of the BIE used, so using our approximation space within a better conditioned BIE such as the second kind formulations proposed in \citet{BrLi:12,LiBr:13} might lead to reduced condition numbers.  

Finally, in the last column of Table~\ref{table1} we show the relative computing time (rel cpt) required for setting up and solving the linear system (we solve the system directly), measured with respect to the time required for $k=10$.  We emphasize the fact that the computing time is independent of $k$, reflecting that all of the integrals are evaluated using Filon quadrature in a $k$-independent way.

We now turn our attention to the approximation of $u(\bx)$, $\bx\in D$, and of the far field pattern $F$ (often the quantities of real interest in scattering problems).  To investigate the accuracy of $u_N(\bx)$, we compute the error in this solution at 89600 evenly spaced points (corresponding to 10 points per wavelength for $k=640$) around the perimeter of the rectangle with corners at $(-\pi,\pi)$, $(11\pi,\pi)$, $(11\pi,-\pi)$, $(-\pi,-\pi)$, which surrounds the screen.  This thus includes points in the illuminated and shadow regions.  To allow easy comparison between different discretizations, noting again that for each example $N$ depends only on $p$, we denote the solution on this rectangle (with a slight abuse of notation) by $u_p(t) := u_N(\bx(t))$, $t\in[0,28\pi]$, where %
$\bx(t)$ represents an arclength parametrisation of the rectangle perimeter.%

In Figure~\ref{fig:domain_errors} we plot for each incident angle and for $k=10$, 40, 160 and 640 the maximum absolute error, $\max_{t\in[0,28\pi]}|u_7(t)-u_p(t)|$, and the relative maximum absolute error, $\max_{t\in[0,28\pi]}|u_7(t)-u_p(t)|/\max_{t\in[0,28\pi]}|u_7(t)|$.
For each example the exponential decay as $p$ increases is clear (note the logarithmic scale on the vertical axes), as predicted by~\rf{GalerkinErrorDomain}.  Moreover, for each fixed value of $p$ the errors appear to be decreasing as $k$ increases; this is better than the 
mild growth with $k$ of the bound~\rf{GalerkinErrorDomain}.  These relative errors are smaller than those on the boundary in Figure~\ref{fig:bdy_errors}.  
\begin{figure}[p]
\begin{center}
\subfigure[non-grazing, $\max_{t\in[0,28\pi]}|u_7(t)-u_p(t)|$]{\includegraphics[width=5.9cm]{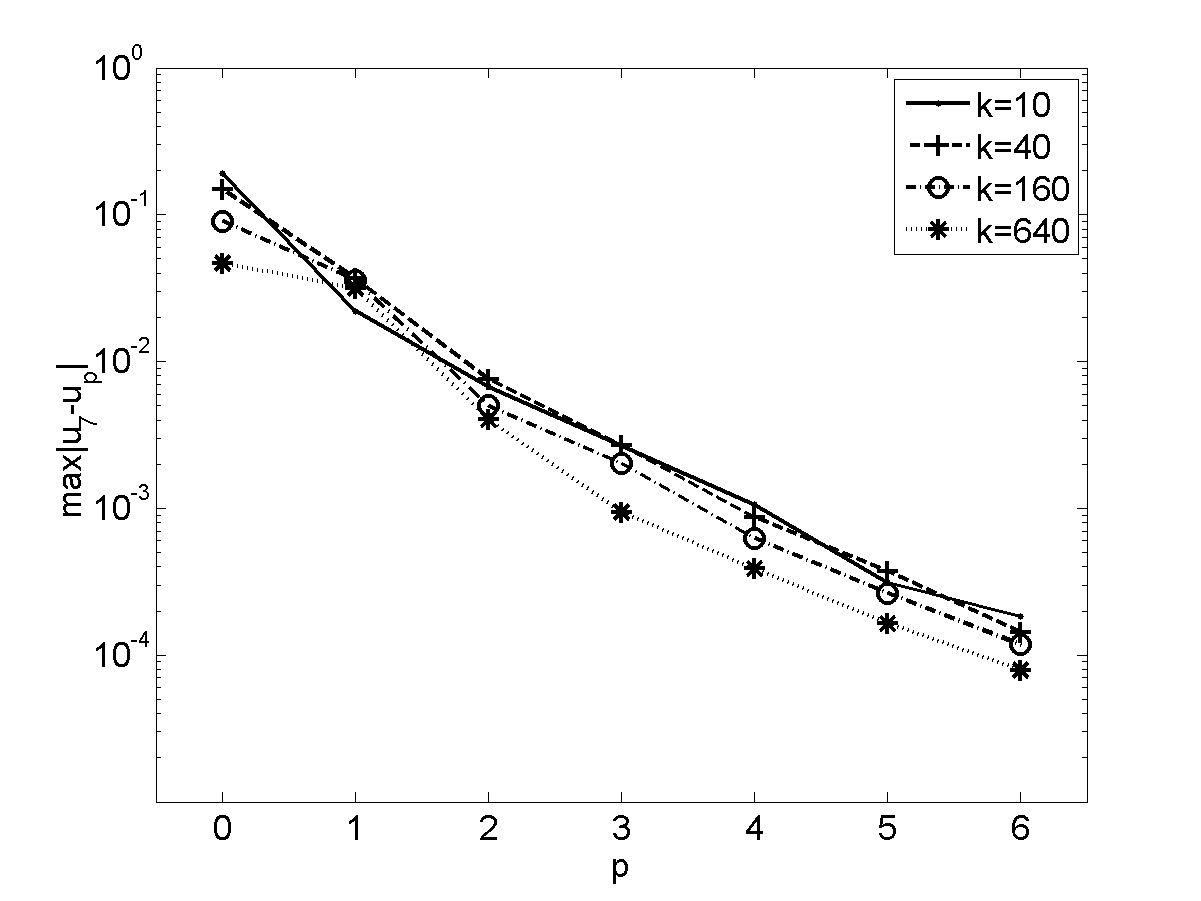}}
\hs{2}
\subfigure[grazing, $\max_{t\in[0,28\pi]}|u_7(t)-u_p(t)|$]{\includegraphics[width=5.9cm]{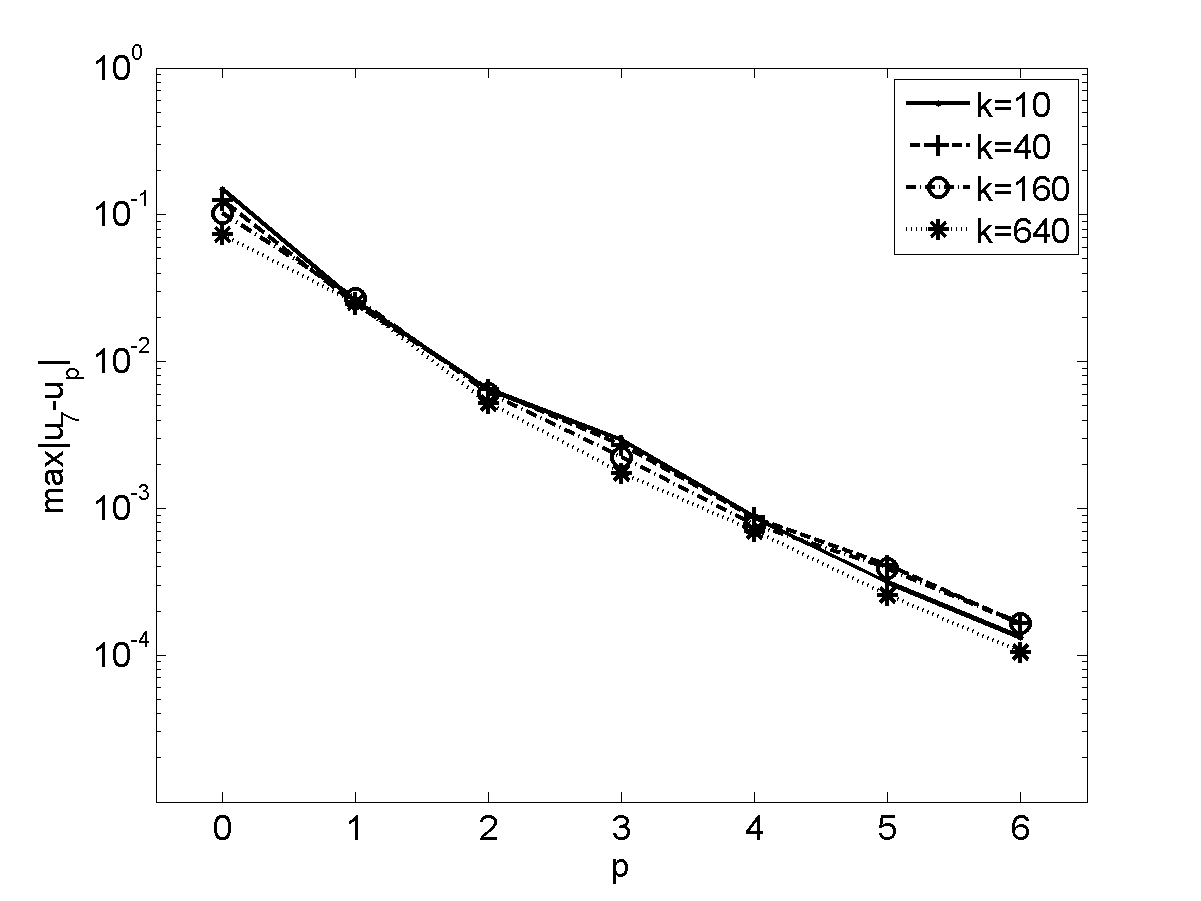}}
\vs{1}
\subfigure[non-grazing, $\frac{\max_{t\in[0,28\pi]}|u_7(t)-u_p(t)|}{\max_{t\in[0,28\pi]}|u_7(t)|}$]{\includegraphics[width=5.9cm]{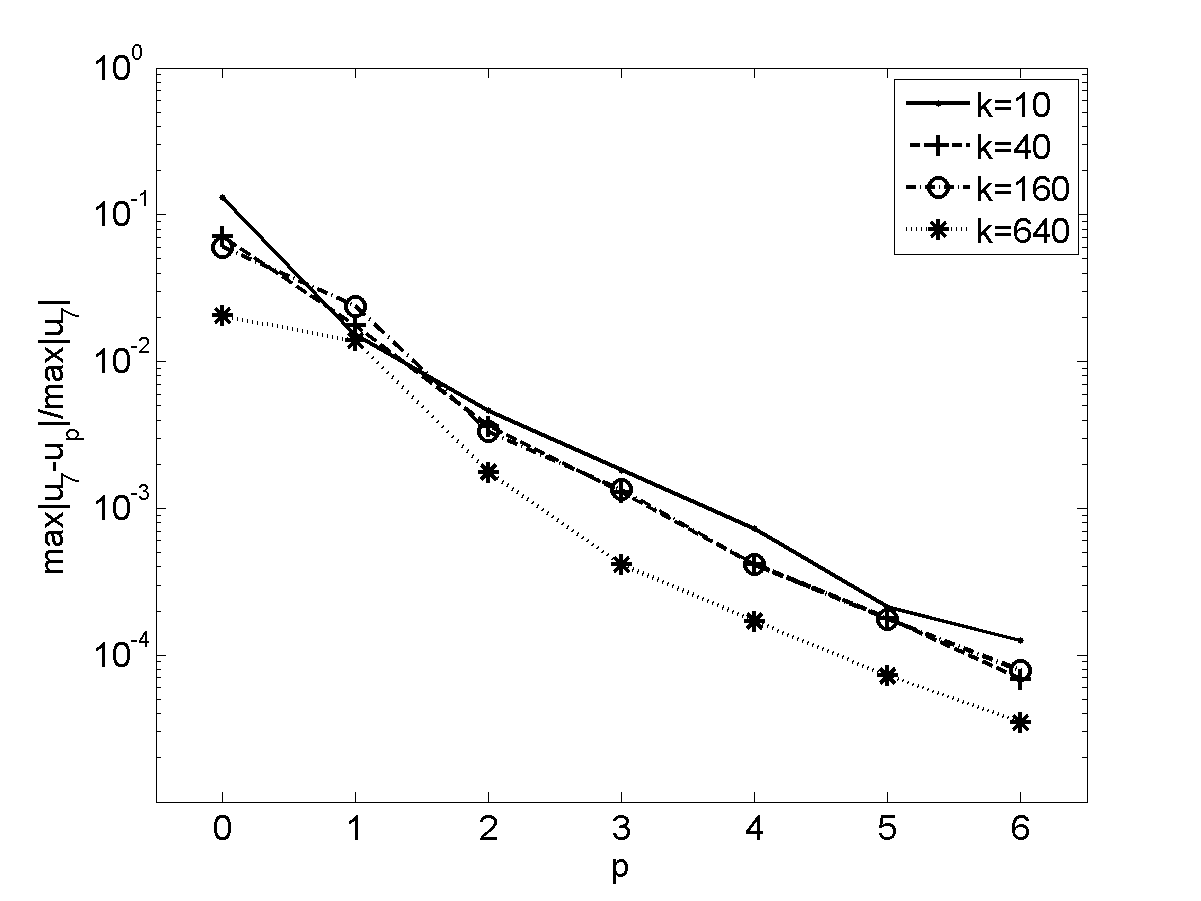}}
\hs{2}
\subfigure[grazing, $\frac{\max_{t\in[0,28\pi]}|u_7(t)-u_p(t)|}{\max_{t\in[0,28\pi]}|u_7(t)|}$]{\includegraphics[width=5.9cm]{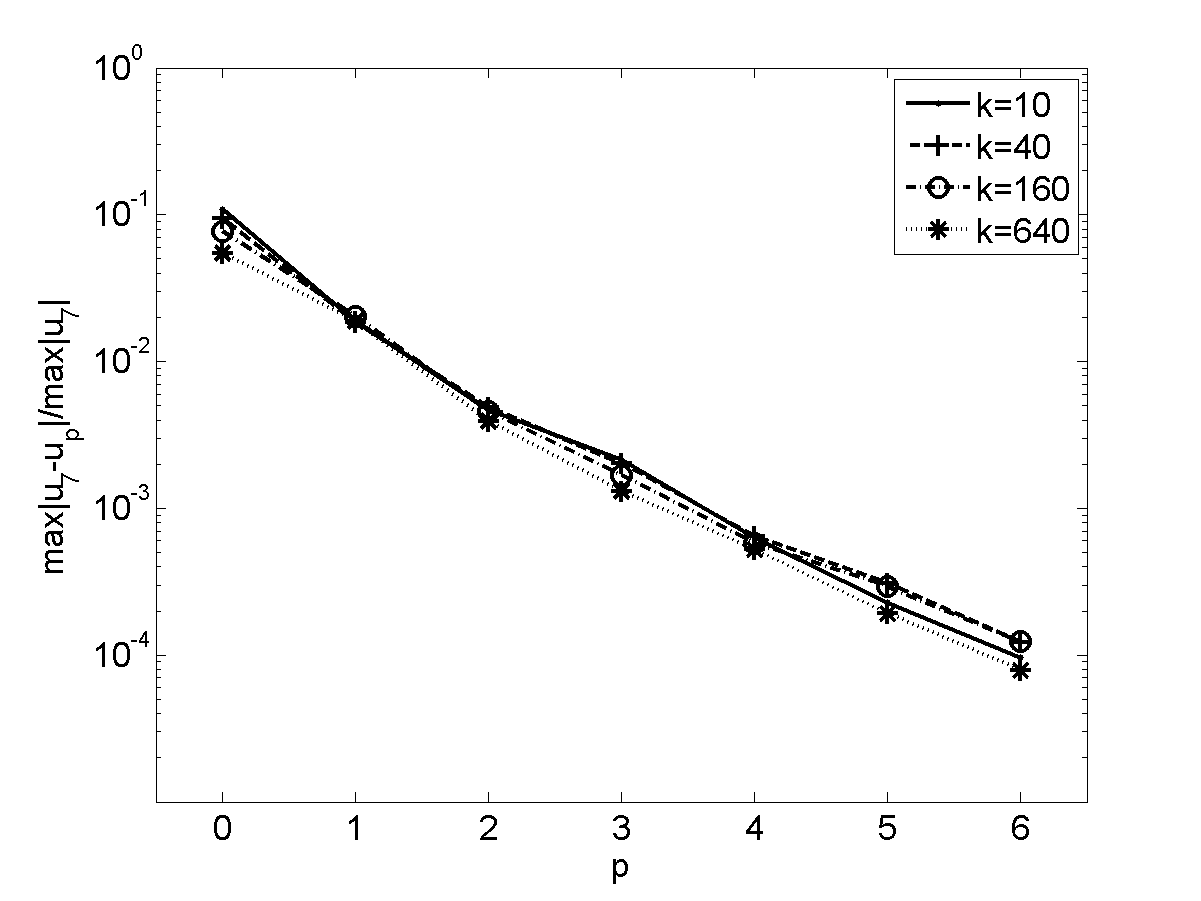}}
\end{center}
\caption{Errors in domain solution on a rectangle surrounding the screen.}%
\label{fig:domain_errors}
\end{figure}

Finally, we compute our approximation~(\ref{eqn:FFP_approx}) to the far field pattern.
Again, %
with a slight abuse of notation, we define
$F_p(t) := F_N(\hat{\bx}(t))$, $t\in[0,2\pi]$,
where $t=0$ corresponds to the direction from which $u^i$ is incident and $\hat \bx(t)$ is a point at angular distance $t$ (measured anticlockwise) around the unit circle.
Plots of $|F_7(t)|$ (the magnitude of the far field pattern computed with our finest discretization) for each of the two incident directions, for $k=1280$, are shown in Figure~\ref{fig:FFP1}. 
For non-grazing incidence, the peaks corresponding to the geometric shadow (i.e.\ the forward-scattering direction) and the specular reflection are indicated (compare Figure~\ref{fig:FFP1} with Figure~\ref{fig:screen_domain}).  We also show the points at which $\hat{\bx}(t)\in\Gamma_\infty$.  For grazing incidence, the shadow peak %
is much lower for than for non-grazing incidence; in the grazing case, there is no reflected peak.
\begin{figure}[p]
\begin{center}
\subfigure[non-grazing]{\includegraphics[width=5.9cm]{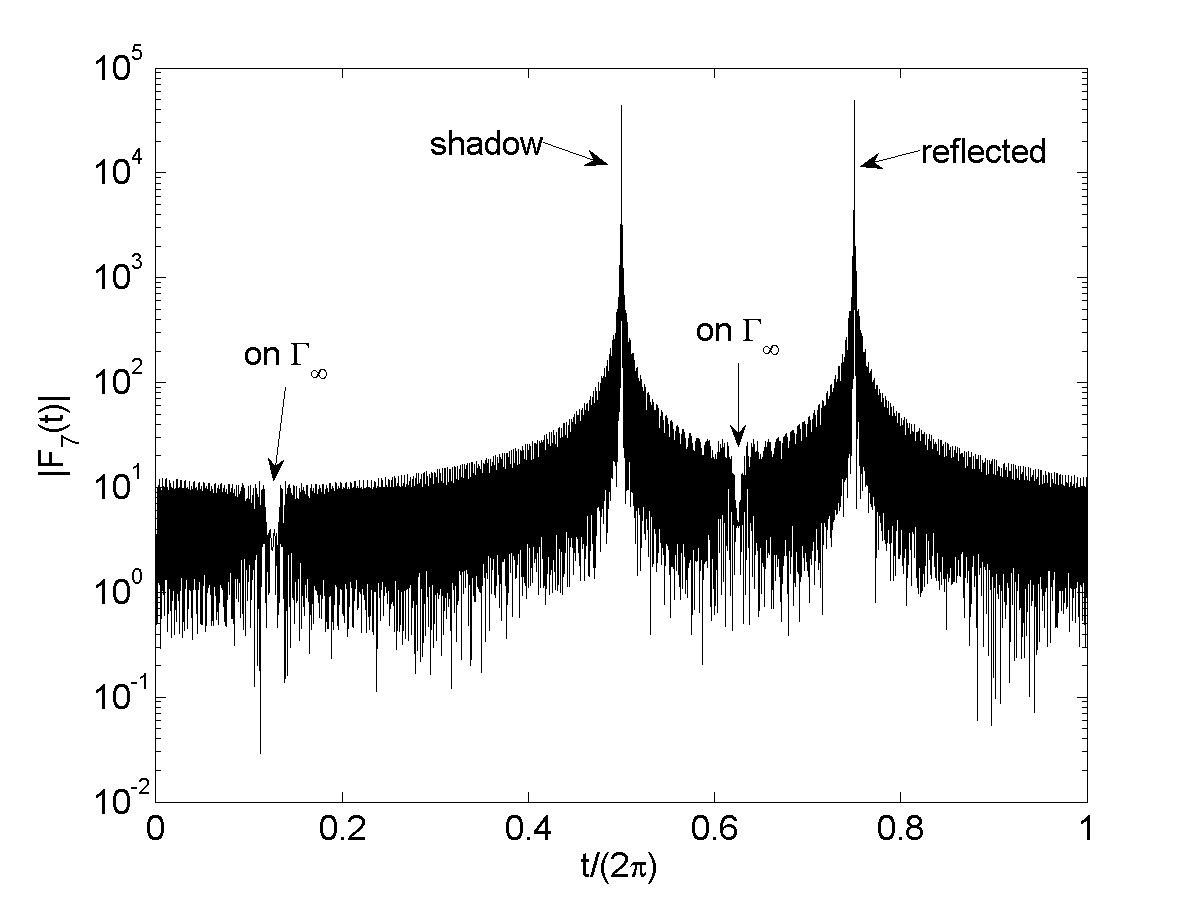}}
\hs{2}
\subfigure[grazing]{\includegraphics[width=5.9cm]{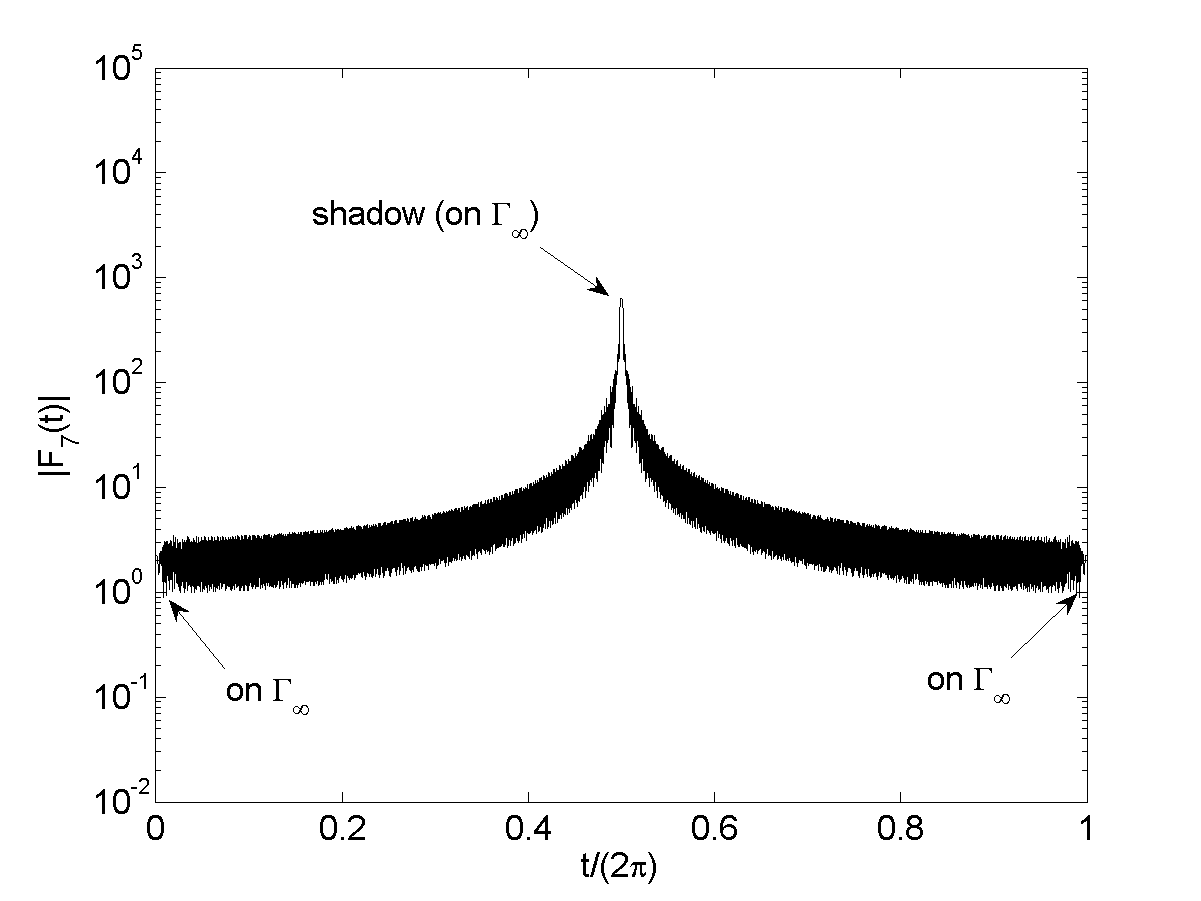}}
\end{center}
\caption{Far field patterns, $|F_7(t)|\approx|F(t)|$, $k=1280$.}
\label{fig:FFP1}
\end{figure}

In Figure~\ref{fig:FFP3} we plot approximations to $\|F_7-F_p\|_{L^{\infty}(\mathbb{S}^1)}$ and $\|F_7-F_p\|_{L^{\infty}(\mathbb{S}^1)}/\|F_7\|_{L^{\infty}(\mathbb{S}^1)}$ %
for $k=20$, 80, 320 and 1280, for each of the two incident directions. 
To approximate the $L^{\infty}$ norm, we compute $F_7$ and $F_p$ at 50,000 evenly spaced points on the unit circle.
The exponential decay as $p$ increases, as predicted by Theorem~\ref{FarFieldThm}, can be clearly seen (again, note the logarithmic scale on the vertical axes).
\begin{figure}[p]
\begin{center}
\subfigure[non-grazing, $\|F_7-F_p\|_{L^{\infty}(\mathbb{S}^1)}$]{\includegraphics[width=5.9cm]{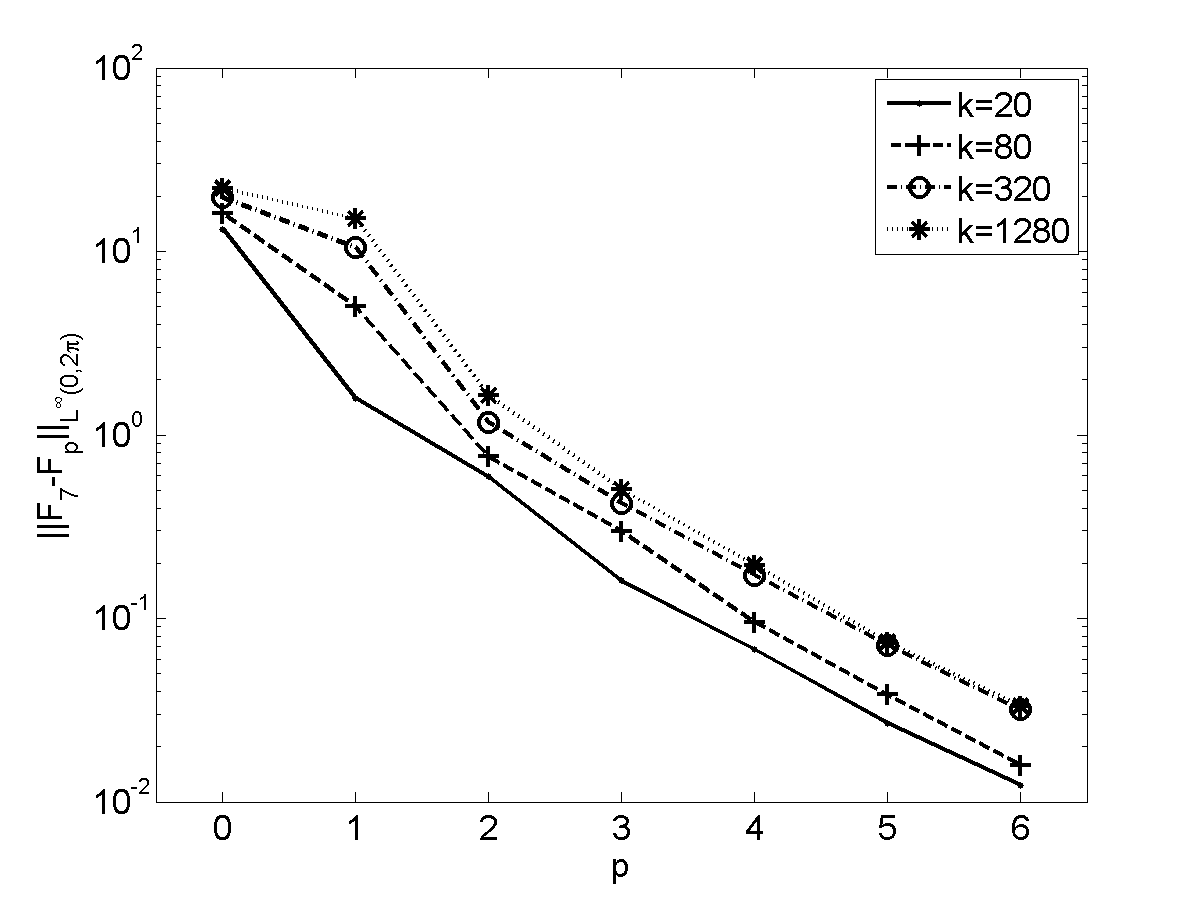}}
\hs{2}
\subfigure[grazing, $\|F_7-F_p\|_{L^{\infty}(\mathbb{S}^1)}$]{\includegraphics[width=5.9cm]{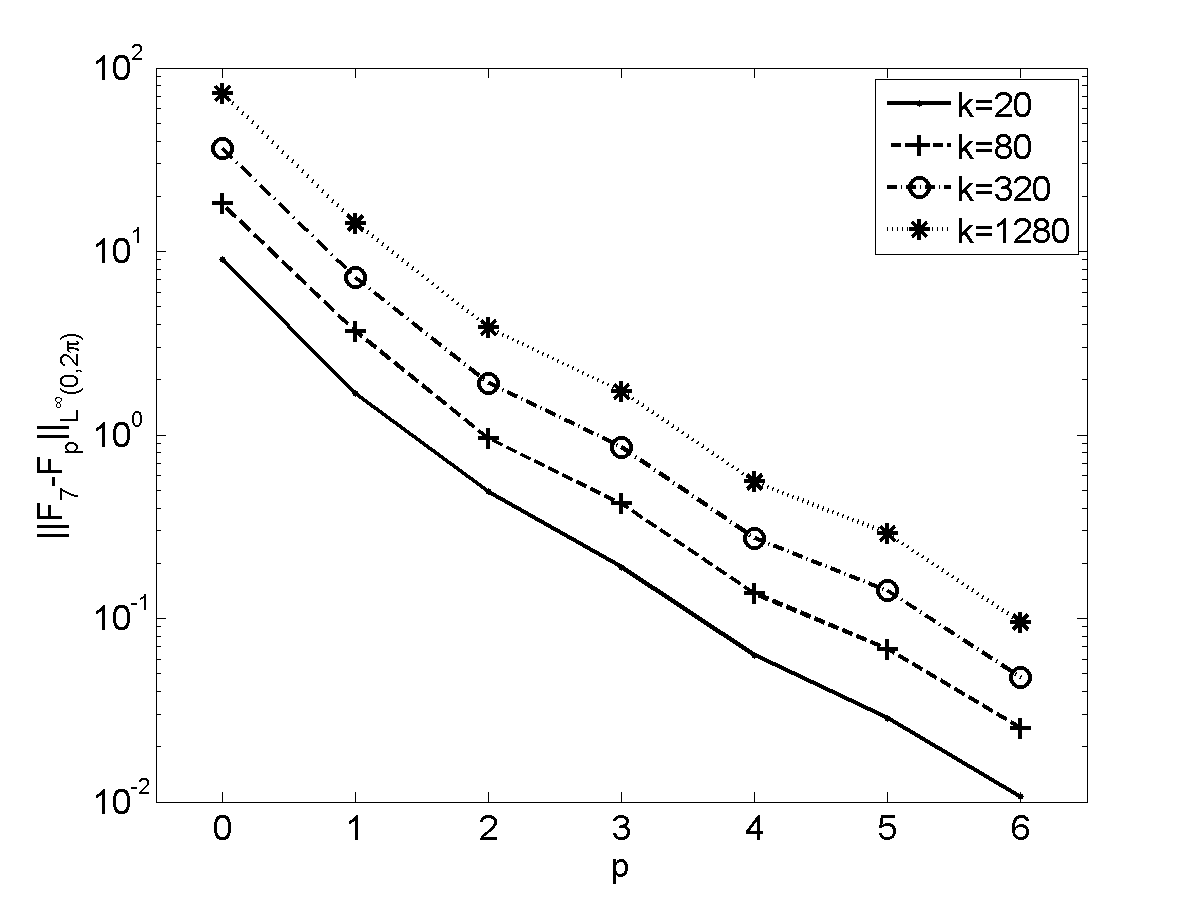}}
\vs{1}
\subfigure[non-grazing, $\displaystyle{\frac{\|F_7-F_p\|_{L^{\infty}(\mathbb{S}^1)}}{\|F_7\|_{L^{\infty}(\mathbb{S}^1)}}}$]{\includegraphics[width=5.9cm]{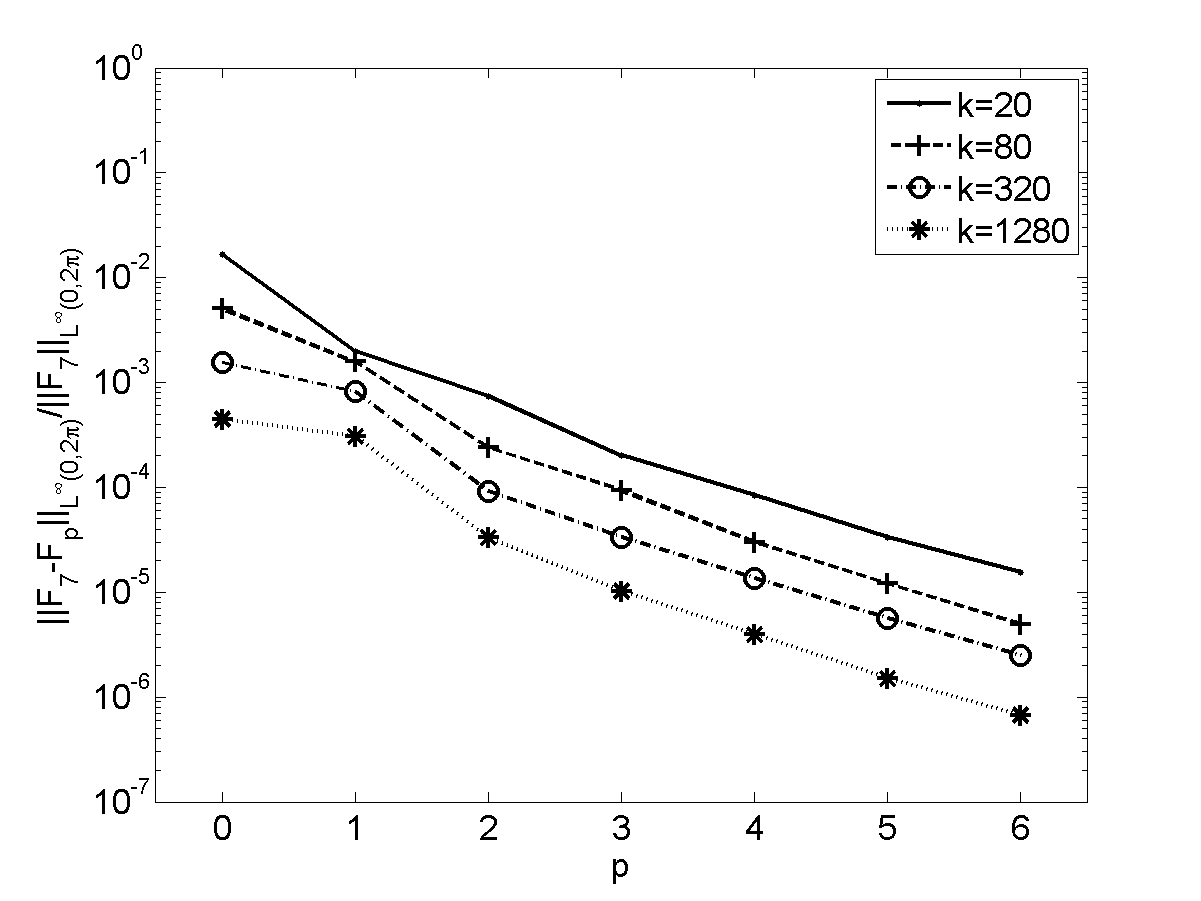}}
\hs{2}
\subfigure[grazing, $\displaystyle{\frac{\|F_7-F_p\|_{L^{\infty}(\mathbb{S}^1)}}{\|F_7\|_{L^{\infty}(\mathbb{S}^1)}}}$]{\includegraphics[width=5.9cm]{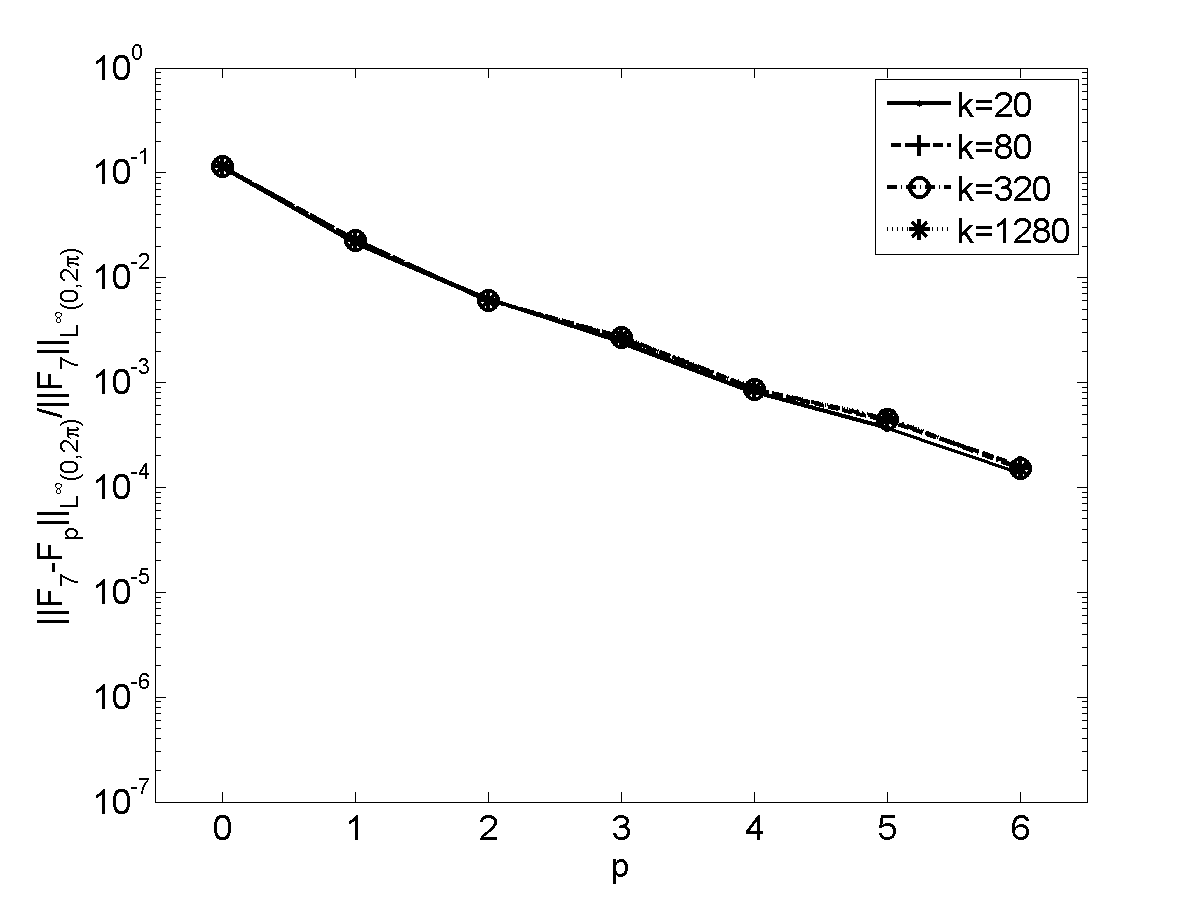}}
\end{center}
\caption{Errors in the far field pattern (note the different scales on the upper and lower plots).}%
\label{fig:FFP3}
\end{figure}

For fixed $p$, the errors $\|F_7-F_p\|_{L^{\infty}(\mathbb{S}^1)}$ increase slowly as $k$ increases.  To investigate this behaviour more carefully, in Table~\ref{table2} we show results for the two angles of incidence for $p=5$ (and hence $N=450$), for a wider range of values of $k$. 
We also tabulate $\log_2(f_p(2k)/f_p(k))$, where $f_p(k)$ refers to the absolute error $\|F_7-F_p\|_{L^{\infty}(\mathbb{S}^1)}$ for a particular value of $k$. This is an estimate of the order of convergence, $\zeta$, on a hypothesis that $f_p(k) \sim k^{\zeta}$ as $k\rightarrow\infty$.  The values of $\zeta\in(0.02,0.75)$ for non-grazing incidence, and $\zeta\approx0.5$ for grazing incidence, are considerably lower than might be anticipated from the estimate~(\ref{FarFieldErrorEst}), suggesting that our estimates are not sharp in terms of their $k$-dependence.  In particular, the results are again consistent with the conjecture that $\uM=\ord{1}$ (as discussed just before Lemma~\ref{lem:Mu}).  In the last column of Table~\ref{table2}, we show how $\|F_7\|_{L^{\infty}(\mathbb{S}^1)}$ grows with $k$.  For non-grazing incidence, $\|F_7\|_{L^{\infty}(\mathbb{S}^1)}$ grows approximately linearly with $k$, and so the relative error $\|F_7-F_p\|_{L^{\infty}(\mathbb{S}^1)}/\|F_7\|_{L^{\infty}(\mathbb{S}^1)}$ decreases as $k$ increases.  For grazing incidence, $\|F_7\|_{L^{\infty}(\mathbb{S}^1)}$ grows almost exactly like $k^{1/2}$, the same rate as $\|F_7-F_p\|_{L^{\infty}(\mathbb{S}^1)}$, and hence the relative error remains approximately constant as $k$ increases.
\begin{table}[htb]
  \begin{center}
    \begin{tabular}{|c|r|c|c|c|c|}
       \hline
      $\bd$ & $k$ & $\|F_7-F_p\|_{L^{\infty}(\mathbb{S}^1)}$ & $\zeta$ & $\frac{\|F_7-F_p\|_{L^{\infty}(\mathbb{S}^1)}}{\|F_7\|_{L^{\infty}(\mathbb{S}^1)}}$ &
       $\|F_7\|_{L^{\infty}(\mathbb{S}^1)}$ \\
      \hline
 $\left(\frac{1}{\sqrt{2}},\frac{-1}{\sqrt{2}}\right)$  
                             &   10 &   1.60$\times10^{-2}$ & 0.75 &   3.99$\times10^{-5}$ &   4.02$\times10^{2}$  \\
        (non-grazing)        &   20 &   2.69$\times10^{-2}$ & 0.28 &   3.36$\times10^{-5}$ &   8.01$\times10^{2}$  \\
                             &   40 &   3.27$\times10^{-2}$ & 0.23 &   2.05$\times10^{-5}$ &   1.60$\times10^{3}$  \\
                             &   80 &   3.84$\times10^{-2}$ & 0.62 &   1.20$\times10^{-5}$ &   3.20$\times10^{3}$  \\
                             &  160 &   5.90$\times10^{-2}$ & 0.27 &   9.23$\times10^{-6}$ &   6.39$\times10^{3}$  \\
                             &  320 &   7.14$\times10^{-2}$ & 0.04 &   5.59$\times10^{-6}$ &   1.28$\times10^{4}$  \\
                             &  640 &   7.33$\times10^{-2}$ & 0.02 &   2.89$\times10^{-6}$ &   2.54$\times10^{4}$  \\
                             & 1280 &   7.43$\times10^{-2}$ & 0.42 &   1.51$\times10^{-6}$ &   4.94$\times10^{4}$  \\
                             & 2560 &   9.95$\times10^{-2}$ &      &   1.12$\times10^{-6}$ &   8.85$\times10^{4}$  \\
   \hline
 $(1,0)$    &   10 &  1.56$\times10^{-2}$ & 0.87 &   2.79$\times10^{-4}$ &   5.61$\times10^{1}$  \\
  (grazing) &   20 &  2.87$\times10^{-2}$ & 0.64 &   3.61$\times10^{-4}$ &   7.93$\times10^{1}$  \\
            &   40 &  4.47$\times10^{-2}$ & 0.61 &   3.98$\times10^{-4}$ &   1.12$\times10^{2}$  \\
            &   80 &  6.80$\times10^{-2}$ & 0.55 &   4.28$\times10^{-4}$ &   1.59$\times10^{2}$  \\
            &  160 &  9.94$\times10^{-2}$ & 0.51 &   4.43$\times10^{-4}$ &   2.24$\times10^{2}$  \\
            &  320 &  1.41$\times10^{-1}$ & 0.53 &   4.46$\times10^{-4}$ &   3.17$\times10^{2}$  \\
            &  640 &  2.04$\times10^{-1}$ & 0.51 &   4.55$\times10^{-4}$ &   4.49$\times10^{2}$  \\
            & 1280 &  2.90$\times10^{-1}$ & 0.51 &   4.57$\times10^{-4}$ &   6.34$\times10^{2}$  \\
            & 2560 &  4.13$\times10^{-1}$ & 0.50 &   4.61$\times10^{-4}$ &   8.97$\times10^{2}$  \\
   \hline
    \end{tabular}
  \end{center}
  \caption{Errors and relative errors in the far field pattern, for various $k$, with $p=5$ (and hence $N=450$).}
  \label{table2}
\end{table}

In summary, our numerical examples demonstrate that the predicted exponential convergence of our $hp$ scheme is achieved in practice.  
The $k$-explicit error bounds~\rf{GalerkinErrorEst_kexplicit}, \rf{GalerkinErrorDomain} and \rf{FarFieldErrorEst} predict at worst mild growth in errors as $k$ increases, which can be controlled by a logarithmic growth in the degrees of freedom $N$.  But our numerical results support the conjecture that this mild growth is pessimistic, and that for a fixed number of degrees of freedom the accuracy of our numerical approximation to the domain solution and the far-field pattern stays fixed or even improves as the wavenumber $k$ increases. 
We suspect that the apparent lack of sharpness in $k$-dependence of our error bounds is due at least in part to a lack of sharpness in $k$-dependence of the estimate in Lemma~\ref{lem:Mu} for $\uM$, of our best approximation estimate~(\ref{BestAppdudn}), and of the quasi-optimality estimate~\rf{eqn:quasi-opt}.

\section*{Acknowledgements}
We gratefully acknowledge the support of EPSRC grant EP/F067798/1, for all authors, and of EPSRC grant EP/K000012/1, for SL.  We thank the anonymous reviewers for their comments, and we thank Nick Biggs, Alexey Chernov, Peter Svensson and Ashley Twigger for helpful discussions.

\bibliographystyle{siam}
\bibliography{BEMbib_short}

\end{document}